\documentclass[oneside,english]{amsart}
\ifx\HCode\UnDef\else\hypersetup{tex4ht}\fi 

\usepackage[T1]{fontenc}

\usepackage{mathtools}

\usepackage{amsthm,amsmath}
\usepackage{amssymb}
\usepackage{esint}
\usepackage[pdfstartview=FitH,colorlinks=true]{hyperref}
\usepackage{datetime}
\usepackage[alphabetic,initials]{amsrefs}
\usepackage{amsaddr}

\usepackage{tikz}
\usetikzlibrary{plotmarks}
\usetikzlibrary{external}
\usetikzlibrary{patterns}
\usepackage{verbatim}
\usepackage{graphicx}

\DeclareMathOperator{\goodb}{GOOD}
\DeclareMathOperator{\badb}{BAD}

\makeatletter
\numberwithin{equation}{section}
\numberwithin{figure}{section}
\theoremstyle{plain}
\newtheorem{thm}{\protect\theoremname}[section]
  \theoremstyle{definition}
  \newtheorem{defn}[thm]{\protect\definitionname}
  \theoremstyle{plain}
  \newtheorem{lem}[thm]{\protect\lemmaname}
  \theoremstyle{plain}
  \newtheorem{prop}[thm]{\protect\propositionname}
  \theoremstyle{remark}
  
  \theoremstyle{plain}
  \newtheorem{cor}[thm]{\protect\corollaryname}
  \theoremstyle{remark}
  \newtheorem{rem}[thm]{\protect\remarkname}
  \theoremstyle{plain}
  
  \theoremstyle{plain}
  \newtheorem{crit}[thm]{\protect\critname}
  
  \theoremstyle{plain}
  \newtheorem*{fact*}{\protect\factname}

\makeatother

  \providecommand{\corollaryname}{Corollary}
  \providecommand{\definitionname}{Definition}
  \providecommand{\lemmaname}{Lemma}
  \providecommand{\notationname}{Notation}
  \providecommand{\propositionname}{Proposition}
  \providecommand{\remarkname}{Remark}
  \providecommand{\theoremname}{Theorem}
  \providecommand{\factname}{Fact}
  \providecommand{\ackname}{Acknowledgement}
  \providecommand{\critname}{Condition}

\begin{document}

\global\long\def\asigma{\mathscr{A}}
\global\long\def\aesigma{\mathscr{A}_{E}}
\global\long\def\plmu{\mu_{\lambda}^{\aesigma,\omega}}
\global\long\def\condmu{\mu_{\omega}^{\aesigma}}
\global\long\def\aroots{\mathcal{A}_{\left\{  \pm1,0\right\}  }}
\global\long\def\singu{S_{\perp}}
\global\long\def\pl{\Phi_{\lambda}}
\global\long\def\lpi{\pi_{\lambda}}
\global\long\def\tmu{\tilde{\mu}}
\global\long\def\Aai{A_{\infty}^{\alpha}}
\global\long\def\Aan{A_{n}^{\alpha}}
\global\long\def\supp{\mathrm{supp}}
\global\long\def\next{\mathrm{next}}
\global\long\def\match{\mathrm{match}}
\global\long\def\suggs{\mathrm{suggs}}
\global\long\def\Leb{\mathrm{\mathscr{L}eb}}
\global\long\def\dLeb{\mathrm{d\mathscr{L}eb}}
\global\long\def\Var{\mathrm{Var}}
\global\long\def\intervals{\mathrm{Int}}
\global\long\def\thetainp{\theta_{\mathrm{input}}}
\global\long\def\thetaoutp{\theta_{\mathrm{output}}}
\global\long\def\Nup{N^{\mathrm{upper}}}
\global\long\def\Nlow{N^{\mathrm{lower}}}
\global\long\def\nup{n^{\mathrm{upper}}}
\global\long\def\nlow{n^{\mathrm{lower}}}
\global\long\def\rwrs{\mathrm{RWRS}}
\global\long\def\red{\mathrm{red}}
\global\long\def\len{\mathrm{len}}
\global\long\def\regwalks{\mathcal{R}}
\global\long\def\walks{\mathcal{W}}
\global\long\def\sinwalks{\walks_{\mathrm{s}}}
\global\long\def\pbin{\Psi_{\mathrm{Bin}}}
\global\long\def\coloneqq{:=}
\global\long\def\regwalksall{\mathcal{R}_{\ge}^{\mathrm{all}}}

\title{Scenery Reconstruction for Random Walk on Random Scenery Systems}

\author{Tsviqa Lakrec}
\address{Einstein Institute of Mathematics\\
               The Hebrew University of Jerusalem\\
                Edmond J. Safra Campus, Jerusalem, 91904, Israel}

\date{\today}\thanks{This work is part of the author's PhD thesis. The author acknowledges support from the ISF grant 891/15 and ERC 2020 grant HomDyn 833423.}
\begin{abstract}
Consider a simple random walk on $\mathbb{Z}$ with a random coloring of $\mathbb{Z}$. 
Look at the sequence of the first $N$ steps taken in the random walk, together with the colors of the visited locations. We call this the record.
From the record one can deduce the coloring of of the interval in $\mathbb{Z}$ that was visited, which is of size approximately $\sqrt{N}$. 
This is called scenery reconstruction. 
Now suppose that an adversary may change $\delta N$ entries in the record that was obtained. What can be deduced from the record about the scenery now?
In this paper we show that it is likely that we can still reconstruct a large part of the scenery.

More precisely, we show that for any $\theta<0.5,\,p>0$ and $\epsilon>0$, there are $N_{0}$ and $\delta_{0}$ such that if $N>N_{0}$ and $\delta<\delta_{0}$ then with probability $>1-p$ the walk is such that we can reconstruct the coloring of $>N^{\theta}$ integers in the scenery, up to having a number of suggested reconstructions that is less than $2^{\epsilon s}$, where $s$ is the number of integers whose color we reconstruct.

\end{abstract}

\maketitle
\tableofcontents{}

\section{Introduction}

Consider a simple random walk over the integers. Denote by $\omega$ the steps of the walk.
We may look at $\omega$ as a series of $\omega_{t}\in\left\{\pm 1\right\}$ 
for $t\in\left\{0,1,\dots\right\}$. Let the integers be colored in a random i.i.d way, 
for some set of colors that we will refer to as the alphabet, $C$. We refer to this random coloring 
as the scenery. Let $X^{\omega}_{t}$ denote the position of the walk after $t$ steps.

We look at the following sequence: for times $t$ from $0$ to some $N$, look at the 
color of the scenery at the position $X^{\omega}_{t}$, and at the direction of the walk 
from this position at that time, $\omega_{t}$. Note that from the record of these values for 
$0\le t < N$, one can deduce what exactly were the first $N$ steps of the walk, and what was 
the coloring of $\mathbb{Z}$ on the interval that the walk visited in the first $N$ steps. 
Informally speaking, this interval would be of length of order $\sqrt{N}$.

The problem we consider is this: An adversary changed $\delta N$ entries in the first $N$ entries 
in the record of colors and directions. We want to get as much information as possible about the 
visited scenery, and the adversary wants to confuse us as much as possible.

In our result, the information we obtain on the scenery is 
delivered in the following form: A set of partial suggestions, each of which describes some part of the walk and visited scenery prior to their modification, such that one of these suggestions agrees with the real scenery and walk
. This is best though of as an algorithm that takes the 
walk-scenery record sequence that had been edited and returns suggestions for possible 
reconstruction of the original walk-scenery sequence.

The dynamical system that produces the record is known as the $T,T^{-1}$ system, or alternatively as the random walk on random scenery system (abbreviation: RWRS). The recovery of information about the scenery from the record 
falls under the general theme of scenery reconstruction.
\begin{defn}
Let $n$ be some positive integer, and let $t$ be an integer between $0$ and $n-1$. 
For any $\omega\in\left\{\pm 1\right\}^{I}$, where $I\subset\mathbb{Z}$ is a possibly infinite interval containing $0$, define:
\begin{align*}
X_{t}^{\omega}&\coloneqq
\begin{dcases}
\sum_{i=0}^{t-1}\omega\left(i\right) &\text{ if } t>0, \\
0 &\text{ if } t=0, \\
-\sum_{i=t}^{-1}\omega\left(i\right) &\text{ if } t<0.
\end{dcases}
\end{align*}
This is called \textbf{the position of the walk $\omega$ at time $t$}.
\end{defn}

\begin{defn}
\label{defn:next}
Let $n,L$ be some positive integers, and let $i$ be an integer between $0$ and $n-1$. 
For any $\omega\in\left\{\pm 1\right\}^{\left\{0,1,\dots,n-1\right\}}$ define:
\[
\next\left(\omega,L,i\right)\coloneqq\inf\left(\left\{ j>i:X_{j}^{\omega}-X_{i}^{\omega}\in\left\{ \pm L\right\} \right\} \cup\left\{ n\right\}\right), 
\]
and for any $\omega\in\left\{\pm 1\right\}^{\mathbb{N}\cup\left\{0\right\}}$ or $\omega\in\left\{\pm 1\right\}^{\mathbb{Z}}$ define:
\[
\next\left(\omega,L,i\right)\coloneqq\inf\left\{ j>i:X_{j}^{\omega}-X_{i}^{\omega}\in\left\{ \pm L\right\} \right\}.
\]
\end{defn}
(see Figure \ref{fig:next2}).
\medskip

\begin{figure}[h]
\tikzsetnextfilename{illustration0}
\begin{tikzpicture}[y=.3cm, x=.3cm,font=\sffamily]
 		\draw[->] (0,0) -- coordinate (x axis mid) (26,0);
    \draw[<->] (0,-6) -- coordinate (y axis mid) (0,6);
                    		\node[left=4cm] at (x axis mid) {$\left(t_0,0\right)$};
	\node[right=4.25cm,above=1.5cm] at (x axis mid) {$\left(\mathrm{next}\left(\omega,L,t_0\right),+L\right)$};
	\node[below=0.25cm,right=1cm] at (x axis mid) {$t$};
	\node[above=2.25cm] at (y axis mid) {$X^{\omega}_{t}$};
	\node[above=1.5cm,left=0.5cm] at (y axis mid) {$+L$};
	\node[below=1.5cm,left=0.5cm] at (y axis mid) {$-L$};
	\draw [fill] (0,0) circle [radius=0.2];
	\draw [fill] (25,5) circle [radius=0.2];

		\draw plot[mark=., mark options={fill=white}] 
		file {next_graph.data};
        
        \foreach \y in {-5,0,...,5}
    	\draw[-,dotted] (0,\y) -- (26,\y);

\end{tikzpicture}
\caption{The function $\next$ from Definition \ref{defn:next}}
 \label{fig:next2}
\end{figure}
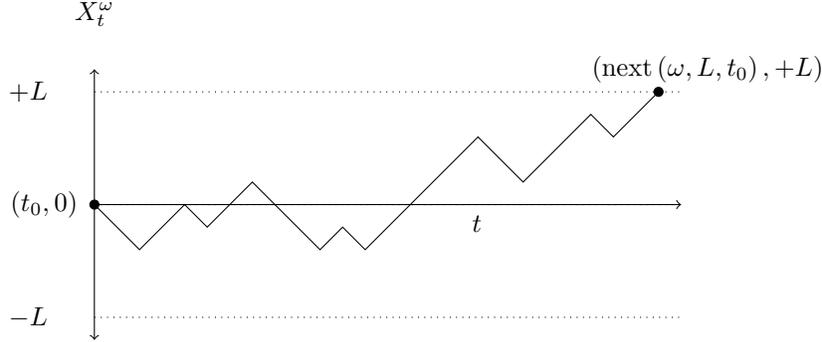

Throughout this paper, we will use a series of parameters $L_{1},\dots,L_{k}$, which are positive integers. The exact values of these parameters will only be given later (see Definition \ref{defn:sinwalks}). 
Until then we will just collect the different constraints that this sequence needs to satisfy. 
The value $k$ is intentionally flexible, so that we will be able to use our method for arbitrarily long walks (which up to technicalities means arbitrarily large $k$). 
How to get the value of $k$ for some length is described in the proof of Theorem \ref{thm:main_theorem}.
\begin{defn}
For our fixed value $k$ and our fixed sequence of $L_{i}$-s, define the set of walks:
\[
\walks\coloneqq\left\{
	\omega \in {\left\{\pm 1\right\}}^{\left\{0,1,\dots,n-1\right\}}:
	\begin{array}{l}
	n\in\mathbb{N},\enskip \left|X^{\omega}_n\right|=L_{1}\cdots L_{k}, \\
	\text{and } 
	\next\left(\omega,L_{1}\cdots L_{k},0\right)=n
	\end{array}
\right\}.
\]
This is the set of finite sequences of $\pm1$-s that end at the first time that they are $L_{1} \cdots L_{k}$ places away from the origin. 
This set is endowed with the probability measure defined by $\Pr\left(\omega\left(i\right)=1\right)=\Pr\left(\omega\left(i\right)=-1\right)=\frac{1}{2}$ independently.
\end{defn}
We also use the following notation:
\begin{defn}
For $\omega\in\left\{\pm 1\right\}^{\left\{0,1,\dots,n-1\right\}}$, denote $\len\left(\omega\right)=n$.
\end{defn}

\begin{defn}
\label{def:test}
For $N,l\in\mathbb{N}$, a \textbf{test over length $N$ of size $l$} is a sequence of tuples $\left(\left(t_{i},\Delta_{i}\right)\right)_{i=1}^{l}$ 
such that for each $i$, $0\le t_{i} <N$ and $-N<\Delta_{i}<N$ are integers.

\noindent A \textbf{test over length $N$} is a test over length $N$ of size $l$ for some $l$, and a \textbf{test} is a test over length $N$ for some $N$. 
The size of a test $\lambda$ is denoted by $\left|\lambda\right|$.
\end{defn}
\begin{defn}
\label{def:passes_test}
Let $\lambda = \left(\left(t_{i},\Delta_{i}\right)\right)_{i=1}^{l}$ be a test over length $N$ of size $l$. 
Let $\sigma\in C^{\mathbb{Z}}$ be a coloring of the integers.
Let $x'\in\left\{\pm1\right\}^{N}\times C^{N}$ be a record of a RWRS, and denote by $\omega'$ its projection to $\left\{\pm1\right\}^{N}$ (the walk record) 
and $\left( x'\left(t\right)_{2} \right)_{t=0}^{N-1}$ its projection to $C^{N}$ (the scenery record).
We say that \textbf{$x'$ passes the test $\lambda$ relative to $\sigma$} if for any $i\in\left\{1,\dots,l\right\}$ and any 
$t\in\left\{t_{i},\dots,\next\left(\omega',L_{1},t_{i}\right)\right\}$ 
it holds that $\sigma\left(X^{\omega'}_{t}+\Delta_{i}\right) = x'\left(t\right)_{2}$.
\end{defn}
If $x'$ passes $\lambda$ relative to $\sigma$, and we know $x'$ and $\lambda$, then this means that we can get partial information on $\sigma$ from $x',\lambda$, or in other words, partially reconstruct $\sigma$ from them. Definition \ref{def:size of reconstructed scenery} measures exactly how much we can reconstruct using $\lambda$ and $x'$.
\begin{defn}
\label{def:size of reconstructed scenery}
Given a test $\lambda=\left(\left(t_{i},\Lambda_{i}\right)\right)_{i=0}^{l}$ over length $N$, and $\omega\in\left\{\pm 1\right\}^N$, the \textbf{size of reconstructed scenery} of $\lambda,\omega'$ is defined by:
\[
\left|\lambda\left(\omega'\right)\right| = 
\left|\bigcup_{i=1}^{l} \left\{X^{\omega'}_{t}+\Delta_{i} : 
	t\in\left\{t_{i},\dots,\next\left(\omega',L_{1},t_{i}\right)\right\}
\right\}\right|.
\]
\end{defn}

Now we can state the main result of this paper:
\begin{thm}[Main Theorem]
\label{thm:main_theorem}
Let $p,\epsilon>0$, $\theta<\frac{1}{2}$. There exist $N_{0}$ and $\delta>0$ such that for any $N>N_{0}$ there exist:
\begin{enumerate}
\item
a set $\sinwalks\subset\left\{\pm 1\right\}^{N}$ with $\Pr\left(\sinwalks\right)>1-p$,
\item
a set $\Lambda$ of tests over length $N$,
\end{enumerate}
so that the following holds. If $x=\left(\left(\omega\left(t\right),\sigma\left(X^{\omega}_{t}\right)\right)\right)_{t=0}^{N-1}\in\left\{\pm 1\right\}^{N}\times C^{N}$, 
where $\omega\in\sinwalks$ and $\sigma\in C^{\mathbb{Z}}$, 
and $x'\in\left\{\pm 1\right\}^{N}\times C^{N}$ is so that 
$\left|\left\{0\le t < N : x\left(t\right) \neq x'\left(t\right) \right\}\right| < \delta N$, 
then there exists $\lambda_{\mathrm{passed}} \in \Lambda$ that satisfies:
\begin{enumerate}
\item
$x'$ passes the test $\lambda_{\mathrm{passed}}$ relative to $\sigma$,
\item
$\left|\lambda_{\mathrm{passed}}\left(\omega'\right)\right| \ge N^{\theta}$, 
where $\omega'$  is the projection of $x'$ to $\left\{\pm 1\right\}^{N}$,
\item
$\ln\left|\Lambda\right| < \epsilon \left|\lambda_{\mathrm{passed}}\left(\omega'\right)\right|$.
\end{enumerate}
\end{thm}

This result can be more clearly understood in these informal terms: 
Let $x$ be a record of the first $N$ steps and colors of a RWRS. For $x$, let $\omega$ be its walk and $\sigma$ the coloring of the scenery $\mathbb{Z}$. 
Our adversary changed $\delta N$ entries from this record, and we now have a corrupted record $x'$.
Then with probability $>1-p$ ($x$ is random, but the adversary's changes aren't) we can reconstruct some part of the scenery by using the ``algorithm'' $\Lambda$: 
Go over all $\lambda\in\Lambda$, and for each one construct a scenery suggestion $\sigma_{\lambda}$ by the formula $\sigma_{\lambda}\left(X^{\omega'}_{t} + \Delta_{i}\right)=x'\left(t\right)_{2}$ for any $i,t$ as in Definition \ref{def:passes_test}. 
The theorem's three conclusions respectively mean that for some $\lambda_{\mathrm{passed}}\in\Lambda$:
(1) The partial reconstruction $\sigma_{\lambda_\mathrm{passed}}$ of $\sigma$ is correct;
(2) We reconstruct over $N^{\theta}$ places in the scenery; and
(3) The number of tests in $\Lambda$ is very small compared to the amount of scenery we successfully reconstruct using $\lambda_{\mathrm{passed}}\in\Lambda$.
Note that as the scenery entropy grows smaller, the probability that $x'$ passes some test in a set of random tests of size $\Lambda$ grows higher. 
This is the motivation for the parameter $\epsilon$, which should be between $0$ and the scenery entropy.

Additionally, $p,\theta$ and $\epsilon$ can be pushed to be arbitrarily close to their limits, which are $p \ge 0,\epsilon\ge0$ and $\theta\le\frac{1}{2}$,
 at the cost of imposing stricter constraints on $N$ and $\delta$. Note that $\theta$ cannot exceed $\frac{1}{2}$, since the law of the iterated logarithm 
 guarantees that for any $\theta'>\frac{1}{2}$ the probability that $\omega\in\left\{\pm 1\right\}^{N}$ visits more than $N^{\theta'}$ places goes to zero as $N$ goes to infinity.

\subsection{Background}

The results of this paper can be viewed as part of the literature on scenery reconstruction, 
and are also closely connected to research in ergodic theory stemming from Kalikow's famous work \cite{Kalikow} on the $T,T^{-1}$ system.

Given a full record $\left(\left(\omega\left(t\right),\sigma\left(X^{\omega}_{t}\right)\right)\right)_{t=0}^{\infty}$ 
of both the random walk, and the scenery as observed along the random walk, 
it is trivial to reconstruct the original scenery, 
as long as the random walk does indeed wander through all sites (an event with probability 1). 
Given a finite segment $\left(\left(\omega\left(t\right),\sigma\left(X^{\omega}_{t}\right)\right)\right)_{t=0}^{n-1}$ 
one can reconstruct $\sigma$ along all sites visited by the random walk up to time $n$, 
which with very high probability contain an interval of size $\sim\sqrt{n}$ around the origin.

The scenery reconstruction problem deals with what can be said if one does not get the full information. 
For instance, there is a large body of literature regarding what happens if one obtains only 
$\left(\left(\sigma\left(X^{\omega}_{t}\right)\right)\right)_{t=0}^{\infty}$, 
without explicitly being given $\omega$. 
From our point of view, errors and omitting the random walk are related obstacles to reconstruction --- 
both deal with having partial information regarding the the full record of the random walk on the scenery $\left(\left(\omega\left(t\right),\sigma\left(X^{\omega}_{t}\right)\right)\right)_{t=0}^{\infty}$. 

The problem of scenery reconstruction as stated above, from a record with omitted walk channel, was the original version formulated by den Hollander and Keane \cite{den_hollander_keane} and by Benjamini and Kesten \cite{benjamini_kesten}.
Note that this question can only be resolved up to a few degrees of freedom, most importantly a reflection of the scenery $\sigma$. It was proven by Lindenstrauss in \cite{lindenstrauss_indistinguishables} that this problem could be resolved at best only up to a set of measure $0$ of sceneries (indeed, it is shown in that paper that there are uncountably many distinct sceneries that are all indistinguishable from each other for reconstruction). We also must get rid of a set of measure 0 of walks (e.g walks that do not visit any integer infinitely many times). The solution of the problem, that reconstruction is indeed possible up to those limits, was given by Matzinger in \cite{matzinger_reconstruction_three} and \cite{matzinger_phd}.

Unlike the regular scenery reconstruction problem, where the record omits the walk channel, the record in this paper does not. Instead, errors are added to the record. When errors can happen, the gap in difficulty between the problem with the steps and without them is less than it may seem, in the following sense. Informally, as the size of the alphabet $C$ grows to infinity\footnote{More precisely, as its entropy grows to infinity.}, the additional information that we get on the scenery from the steps becomes less important. For example if we see in the record  somewhere three consecutive different colors, then it is likely that the walk there was two steps in the same direction, and if we see a color recurring in close indices $i,i+2$ in the record, then it is likely that the steps went in opposite directions. A requirement of a large alphabet is a familiar assumption in the field of scenery reconstruction, e.g \cite{lowe_matzinger_2002_background}. Extension of Theorem \ref{thm:main_theorem} in this direction would be interesting, but not necessary for the motivating problem of this paper from ergodic theory.

Reconstruction method that tolerate errors in the observation record in the classical problem were studied extensively in the last 3 decades.
In \cite{matzinger_rolles_2003b_background},  Matzinger and Rolles showed that reconstruction is possible with probability $1$, if there are random i.i.d errors in the record, with sufficiently low error rate (and also allowing for bounded jumps other than $\pm1$). 
Additionally, in \cite{lember_matzinger_background} Lember and Matzinger  consider random walks in which the step sequence $\omega$ can contain steps other than $\pm1$ or $0$
and show that it is possible to recover a so-called fingerprint of the scenery\footnote{That is, some datum which is in most cases equal for a pair of records only when their sceneries are the same.}. 
A step of $+2$ for example, can be thought of as a deletion of a single entry, and greater steps are deletions of more entries. The step distribution can be chosen so that Lember and Matzinger's result becomes a result on random i.i.d deletions from the record of a simple random walk\footnote{in a subsequent work, \cite{lowe_matzinger_merkl_background}, Lember, Matzinger and Merkl showed it for another type of step distribution in the random walk that is not quite interpretable as a random error.}.

However, the aforementioned results are always based on the assumption that the errors are random, whereas we consider the harder problem of allowing an adversary to cunningly introduce errors in the worst possible way, only restricting the number of errors as a proportion of the record. 
We want to be able to reconstruct some scenery, even if the errors are placed in the most irksome and confusing places for us. As far as we know, scenery reconstruction with this assumption was not explicitly considered until now, but it was indeed implicitly handled in \cite{Kalikow}, 
which also handled adversarial deletions in addition to adversarial errors. Incorporating deletions is somewhat natural for the RWRS we consider, that includes the walk channel $\omega$ which can also contain errors, but was avoided in this paper since it is not conceptually different, but somewhat more technical. Worst-case or adversarial errors are also a topic considered in coding theory \cite{guruswami}.
Because of the different assumption on the errors, the previously cited reconstruction with errors papers are not applicable for our problem. 

Another aspect of the main result of this paper, is that from a finite record of size $N$, we reconstruct a part of the scenery of size $N^\theta$. The original scenery reconstruction problem assumed that the record is infinite, but later finite records were also considered. 
In  \cite{matzinger_rolles_2003a_background} and \cite{matzinger_rolles_2006_background}, for example, the records are finite. Since it is always possible that the finite random walk was particularly unlucky, these results have to make an additional assumption of the probability of successful reconstruction being close to $1$, as we also did.
For more details on scenery reconstruction, and further references, the reader may consult the review paper by den Hollander and Steif \cite{denHollander_Steif_review}.

From the ergodic theoretic perspective, a random walk on a random scenery can be seen as a probability measure preserving system. Recall that a measure preserving system is said to be \emph{Bernoulli} if it is equivalent to an i.i.d. process equipped with the shift map.
Bernoulli systems have very strong mixing property. Another very strong mixing condition, that is still weaker than being Bernoulli system, is for a system to be a \emph{K-system} (or Kolmogorov system; this class of systems is also known as systems with \emph{Completely Positive Entropy}). 
Meilijson proved in \cite{Meilijson} that the RWRS system is a K-system. 
In the early years of ergodic theory, it was an open question whether any K-system is also Bernoulli. 
In \cite{Kalikow}, Kalikow showed that the RWRS system, also known as the $T,T^{-1}$ system in ergodic theory, is not Bernoulli, proving that this natural system is a counter-example for that question. Note that prior to Kalikow's work a K-system that is not Bernoulli was constructed by by Ornstein in \cite{ornstein}; the suggestion that $T,T^{-1}$ might be a counter-example was made by Benjamin Weiss in \cite{benjiweiss}.

Kalikow's proof can be thought of as an application of a certain scenery reconstruction method that he introduces.
He shows that the $T,T^{-1}$ system does not satisfy a property that all systems isomorphic to a Bernoulli system must satisfy, called \emph{Very Weak Bernoulli (VWB)}. 
The VWB property for a probability measure preserving shift system $(Q^{\mathbb Z},T,\mu,\mathcal{B})$ for a finite set $Q$ says the following: 
Using the measure disintegration theorem on the measure $\mu$ for the map $Q^{\mathbb Z}\rightarrow Q^{\left\{\dots,-2,-1\right\}},\;(\omega_i)_{i\in\mathbb Z}\mapsto(\omega_i)_{i<0}$, we get a measurable map $c\mapsto\mu_c$ that takes $(c_i)_{i\in\mathbb Z}\in Q^{\mathbb Z}$ to a probability measure on $Q^{\mathbb Z}$ that depends only on $(c_i)_{i<0}$ and is supported on the subset $S_c = \left\{\omega\in Q^{\mathbb Z}:\forall i<0,\enskip \omega_i=c_i\right\}$, such that $\mu = \int \mu_c d \mu$. We call $\mu_c$ the conditional probability measure on the set $S_c$.
Now pick independently according to $\mu$ two sequences $a,b\in Q^{\mathbb Z}$,
and let $\mu'_a=\pi_{*}\mu_a,\mu'_b=\pi_{*}\mu_b$ be the measures on $Q^n$ we get by
pushing forward $\mu_a,\mu_b$ using the map $\pi:Q^{\mathbb Z}\rightarrow Q^n$ taking $(\omega_i)_{i\in\mathbb Z}$ to $(\omega_i)_{i=0}^{n-1}$. The $\bar{d}$ metric is a metric between probability measures on $Q^n$ defined as the infimum of the expectation of the Hamming distance between $x,y\in Q^n$, where $(x,y)$ is distributed according to some joining $\gamma$ of the measures.
The VWB property says that for any $\epsilon>0$, the probability that the aforementioned $a,b$ satisfy that $\bar{d}(\mu'_a,\mu'_b)>\epsilon$ goes to $0$ as $n$ goes to infinity.

Consider the VWB property in the context of the RWRS system, where $Q=\left\{\pm1\right\}\times C$. Then with probability $1$, the conditional measures on the sets $S_c$ can be thought of as conditional measures on the set of measures with the same scenery as $c$, and the same walk for $i<0$. Thus, the measures $\mu'_x,\;x=a,b,\;$ are the distributions of the first $n$ entries in the record of a RWRS for the fixed scenery~$c$.
Now assume that $\bar{d}(\mu'_a,\mu'_b)<\epsilon$. Then for most records $x$ of length $n$ in the support of $\mu'_a$, we get a that by inserting at most $O(\epsilon)n$ errors we can get a record $y$ in the support of $\mu'_b$, through some joining of them. 
Theorem \ref{thm:main_theorem}, or a weaker analogue of it that is implicitly given in \cite{Kalikow}, implies that the scenery of $x$, which is $a$, needs to be related to the scenery of $y$, which is $b$. But the probability of a random pair $(a,b)$ satisfying this is low, in contradiction, and hence the $T,T^{-1}$ system is not VWB.
The method in Kalikow's proof was further developed to $d$-dimensional walks by den Hollander and Steif \cite{denHollander_Steif_kalikowism}, and to a certain class of smooth systems by D. Rudolph in~\cite{rudolph}.

The particular problem that we look into in this paper, of scenery reconstruction with adversarial errors, is also related to (and in fact was motivated by) recent work of Austin \cite{Austin} answering a longstanding question from ergodic theory about RWRS systems: is the entropy of the scenery of a RWRS system invariant under isomorphism of measure preserving systems? Let us describe this question in some more detail.
Given two RWRS system where the sceneries have the same Kolmogorov-Sinai entropy, we know by Ornstein isomorphism theorem that there exist an isomorphism between the scenery systems, and from it we can make an isomorphism of the RWRS systems.
It is natural to ask whether the converse is true --- that is, given an isomorphism of measure preserving system between two RWRS systems, is the entropy of the sceneries of the two systems necessarily the same? This was answered in the affirmative in full generality by Austin in \cite{Austin}, and prior to that under additional assumptions by Aaronson in \cite{Aaronson}. 
It seems plausible that a stronger form of a scenery reconstruction algorithm under adversarial errors than the one we give in this paper could give a more direct approach to Austin's theorem; this line of attack is suggested by Austin in subsection 4.3 of \cite{Austin}.

\subsection{Overview of the Proof}\label{subsection:proof_overview}

The statement of Theorem \ref{thm:main_theorem} is that typically we can find a set of tests, not too big, so that one of them instructs us how to place an error-free chunk of the record into the scenery, allowing us to reconstruct that part of the scenery from the record. 

The set of tests is generated in a way inspired by Kalikow's proof, but more efficient. The idea is to look at the record through a hierarchy of scales. In our view, it presents an algorithm for reconstruction, and Theorem \ref{thm:main_theorem} shows that this algorithm succeeds with high probability in the walk (though the computational complexity of this algorithm appears to be rather poor). Suppose that we look at a part of the record of size $N$ that covers a scenery interval of size $L$, and has $p N$ adversarial errors in it. 
Then it is clear that for some $N'<N,L'<L,p'\approx p$ there exist a sub-interval of this part of the record which has size $N'$, covers a scenery interval of size $L'$, and has $p' N'$ adversarial errors in it. Kalikow noted that even for the worst possible errors, for a typical walk, there is a pair of such sub-intervals of the record, which have disjoint scenery intervals. 
In this paper, we improve the efficiency of this by packing many more sub-intervals with disjoint scenery intervals into this, so that we are able to cover $>L^{1-o(1)}$ of the scenery. Now, we can use a reconstruction method for the sub-intervals, thus reconstructing part of the scenery, and then trying every possible way to place them all together in our scenery --- it turns out, that we do not get too many possibilities for the reconstructed scenery. The reconstruction of the sub-interval is done in the same way, just with smaller numbers, and we continue so on until we get to the lowest level, when the rate of errors is $0$ and reconstruction of the sub-interval is easy.

The precise formulation of the above ``algorithm'' appears in Definition \ref{def:Lambda} of the sets $\Lambda_m$, and later in section \ref{section_Lambda} it is also proved that for a typical walk we can reconstruct a significant part of the scenery, specifically the color of $N^\theta$ of the visited integers for $\theta<0.5$ arbitrarily close to $0.5$. By a typical walk, we specifically mean any walk  $\omega\in\sinwalks$, which is defined at length in section \ref{section_sinwalks}, culminating in Definition \ref{defn:sinwalks}, and is finally shown to be a set of arbitrarily high probability in Proposition \ref{prop:prob_of_sinuosity}. The definition is quite involved, since it will be used in each level of the multilevel analysis in the ``algorithm'' above. In each level, we require several assumptions to hold, including that the size of the scenery interval covered during a record sub-interval interval of that level's size is approximately the square root of that sub-interval's size, and that the local time measure of the random walk in this sub-interval is never of unusually high value, and also the recursive definition that sufficiently many lower level sub-intervals of this sub-interval have the same typicality properties. The local time measure of a walk $\omega\in\left\{\pm1\right\}^{\left\{0,1,\dots,n-1\right\}}$ is a probability measure on $\mathbb{Z}$ that gives an integer $x$ the probability $\frac{1}{n+1}\left|\left\{0\le t\le n: X^{\omega}_t = x\right\}\right|$.

The proof of Theorem \ref{thm:main_theorem} for a walk in $\sinwalks$ appears in section \ref{section_pf_of_main_thm}, and has three parts:
\begin{enumerate}
\item The proof that some test $\lambda_\mathrm{passed}\in\Lambda$ is satisfied, assuming that the walk is in $\sinwalks$, appears in Lemma \ref{prop:existence_of_satisfied_test} and in subsection \ref{section_some_test_satisfied}. This is item (1) in Theorem \ref{thm:main_theorem}.
\item The proof that the set of tests $\Lambda$ is not too big is given in Proposition \ref{prop:Bound_on_size_of_lambda}, and touched again in subsection \ref{section_some_test_satisfied}. This is item (3) in Theorem \ref{thm:main_theorem}.
\item The proof that the test reconstructs a lot of scenery is given in Lemma \ref{lem:how_much_psila}, and also touched again in subsection \ref{section_some_test_satisfied}. This is item (2) in Theorem \ref{thm:main_theorem}.
\end{enumerate}
A more technical overview of the proof of \ref{thm:main_theorem}, elaborating on this last three points, will be given in subsection \ref{subsection_def_of_tests}.

For the first-time reader, we recommend to first read section \ref{section_sinwalks} without reading the Lemmas and the proofs, then read Proposition \ref{prop:prob_of_sinuosity}, then read section \ref{section_Lambda}, and finally return to read \ref{section_sinwalks} in detail. This will allow the reader to understand the motivation for the definitions in section \ref{section_sinwalks} first, and see the final proof, before getting into the probabilistic arguments as to why the properties in these definitions are expected to hold with high probability.

\subsection*{Acknowledgement}
This work is part of my PhD thesis, done under the direction of Elon Lindenstrauss who suggested the problem of scenery reconstruction 
and  assisted with the proof. 

Many arguments in Section~\ref{section_sinwalks} were suggested by Ori Gurel-Gurevich, 
particularly in the proofs of Lemmas \ref{lem:prob lower bound on next of walk} and \ref{lem:0inbadlocal_prob}. 
Additional useful suggestions were made by Benjy Weiss, Jon Aaronson, Zemer Kosloff, Tim Austin and Ohad Noy Feldheim.

I would also like to thank the anonymous referee for their peer review and helpful suggestions.  

\section{Construction of the Set \texorpdfstring{$\sinwalks$}{W s}}
\label{section_sinwalks}
\label{section_prob_sinuous}

In this section we define the set $\sinwalks$ from Theorem \ref{thm:main_theorem}, and bound from above $\Pr\left(\omega\notin\sinwalks\right)$.

To say whether $\omega$ is in $\sinwalks$, we pick some $L$, 
and divide $\omega$ as a sequence to intervals that begin when $X^{\omega}_{t}$ is a multiple of $L$, and end when it reaches a different multiple of $L$. This process is done for several values of $L$, a sequence of $L$-s in which each term in it is a multiple of the previous term in the sequence. As we shall define throughout Section \ref{section_sinwalks}, $\omega\in\sinwalks$ if these sequences of intervals satisfy a long list of typicality conditions.

Our construction 
is inspired by \cite{Kalikow}: 
in that paper, Kalikow uses a reconstruction that begins by dividing $\left\{0,\dots,\len\left(\omega\right)-1\right\}$ 
to intervals of a constant size and showing that for a typical walk one can find two such intervals, over which the walk is typical, errors are few, and the scenery intervals 
visited at these time intervals are disjoint. To enable us to easily see this disjointness, we divide the scenery to intervals of constant size instead, and therefore we get that the size of the intervals of time we get from it is variable (these intervals are denoted below as $I^{m'}_{m}\left(\omega\right)$). This is the origin of the function $\next$ from Definition \ref{defn:next}.

\begin{defn}
\label{defn:redwalk1}
Let $N,L$ be positive integers, and let $\omega\in\left\{\pm1\right\}^{\left\{0,1,\dots,N-1\right\}}$ be a sequence such that $L$ divides $X_{N}^{\omega}$. 
We define $i_{\omega,L}$, \textbf{the $L$-reduced embedding in $\omega$},  recursively:
\begin{multline*}
\begin{aligned}
i_{\omega,L}\left(0\right) \coloneqq & 0, \\
i_{\omega,L}\left(j+1\right) \coloneqq & \next\left(\omega,L,i_{\omega,L}\left(j\right)\right),
\end{aligned}
\end{multline*}
for any $j$ until we have $i_{\omega,L}\left(j\right)=N$.
\end{defn}
\begin{defn}
\label{defn:redwalk2}
Let $N,L$ be positive integers, and let $\omega\in\left\{\pm1\right\}^{\left\{0,1,\dots,N-1\right\}}$ be a sequence such that $L$ divides $X_{N}^{\omega}$. 
We define $\red\left(\omega,L\right)$, \textbf{the $L$-reduced embedded walk of $\omega$} as the sequence 
$\omega_{2}\in\left\{\pm1\right\}^{\left\{0,1,\dots,m-1\right\}}$ such that for $j\in\left\{0,1,\dots,m-1\right\}$ we have
\[
\omega_{2}\left(j\right) \coloneqq \frac{X^{\omega}_{i_{\omega,L}\left(j+1\right)}-X^{\omega}_{i_{\omega,L}\left(j\right)}}{L},
\]
where $\left\{0,1,\dots,m-1\right\}$ is the domain of the $L$-reduced embedding in $\omega$. 
\end{defn}
Note that $m=\len\left(\red\left(\omega,L\right)\right)$ and that $i_{\omega,L}$ is a monotone increasing function from $\left\{0,\dots,\len\left(\red\left(\omega,L\right)\right)-1\right\}$ to $\left\{0,\dots,N-1\right\}$.
Definitions \ref{defn:redwalk1} and \ref{defn:redwalk2} could be extended to infinite $\omega$-s and to negative $j$-s but this is unnecessary for our goals.

\begin{defn}
\label{defn:redwalk3}
For a finite walk $\omega\in\walks$ and an integer $0\le m \le k$ denote:
\[
\len_{m}\left(\omega\right) \coloneqq \len\left(\red\left(\omega,L_{1}\cdots L_{m}\right)\right).
\]
\end{defn}

Now we state the first of many typicality conditions that a walk $\omega\in\sinwalks$ must satisfy: $\omega\in\regwalksall$. This condition will be expressed using positive integer parameters $N_{m}$, that will be fixed later at Definition \ref{defn:sinwalks}.
\begin{defn}
\label{defn:r_all}
We denote the set of walks which have level-$m$-reduced walks of length at least 
$N$ by:
\[
\regwalks_{\ge N}^{m}\coloneqq\left\{ \omega\in\walks:\len_{m}\left(\omega\right)\ge N\right\} .
\]

Usually we will have $N=N_{m}$, and then we will define:
\[
\regwalks_{\ge}^{m} \coloneqq \regwalks_{\ge N_{m}}^{m}.
\]

Additionally, denote $\regwalksall \coloneqq \cap_{m=0}^{k}\regwalks_{\ge}^{m}$.
\end{defn}
Note that $N_{0}$ of Definition \ref{defn:r_all} is not related to $N_{0}$ of Theorem \ref{thm:main_theorem}. 
\begin{defn}
\label{defn:intervals}
We define \textbf{the $j$-th interval of level $m$ at ground level} 
of a walk $\omega$ as:
\[
I_{m,j}^{0}\left(\omega\right)\coloneqq\left\{ i_{\omega,L_{1}\cdots L_{m}}\left(j\right),\dots,i_{\omega,L_{1}\cdots L_{m}}\left(j+1\right)-1\right\}. 
\]
More generally, for $m'<m$ we define 
\textbf{the $j$-th interval of level $m$ at level $m'$} of a walk $\omega$ as:
\[
I_{m,j}^{m'}\left(\omega\right)\coloneqq\left\{ i_{\red\left(\omega,L_{1}\cdots L_{m'}\right),L_{m'+1}\cdots L_{m}}\left(j\right),\dots,i_{\red\left(\omega,L_{1}\cdots L_{m'}\right),L_{m'+1}\cdots L_{m}}\left(j+1\right)-1\right\}. 
\]
We particularly denote $I_{m,j}\left(\omega\right)\coloneqq I_{m,j}^{m-1}\left(\omega\right)$
and this is called \textbf{the $j$-th interval of level $m$}.
In the case of $m'=m$ we simply define $I_{m,j}^{m'}\left(\omega\right)\coloneqq\left\{j\right\}$.
\end{defn}

Figure \ref{fig:red_walk} illustrates the last few definitions for $L=L_{1}\dots L_{m}$:
\begin{figure}[h]
\tikzsetnextfilename{illustration1_new}
\begin{tikzpicture}[y=.13cm, x=.024cm,font=\sffamily]
    \draw[->] (0,0) -- coordinate (x axis mid) (335+15,0);
    \draw[<->] (0,-15) -- coordinate (y axis mid) (0,20);
    	\foreach \y in {-2,-1,...,3}
     		\draw (1pt,5*\y) -- (-3pt,5*\y) 
     		    node[anchor=east] {\y $L$};
    \node[right=4.5cm] at (x axis mid) {$t$};
	\node[above=2.5cm] at (y axis mid) {$X^{\omega}_{t}$};
	\draw plot[mark=., mark options={fill=white}] 
	    file {div_soft_short.data};
    \foreach \y in {-15,-10,...,15}
    	\draw[-,dotted] (0,\y) -- (335,\y);
     \draw[pattern = north east lines]  (0,0+5) rectangle (11, -5)  (56, 0+5) rectangle (75, -5)  (100, 0-5) rectangle (105, 5)  (144, 10+5) rectangle (173, 5) (222, 0+5) rectangle (263, -5)  (304, 0-5) rectangle (315, 5)  (330, 10-5) rectangle (335, 15)  ;          
     \draw[pattern = north east lines] (11,-5-5) rectangle (56, 0)  (75, -5-5) rectangle (100, 0)  (105, 5-5) rectangle (144, 10)  (173, 5+5) rectangle (222, 0)  (263, -5-5) rectangle  (304, 0)  (315, 5-5) rectangle  (330, 10);     \draw[->,blue,mark=x,thick] (0, 0) -- (11, -5) -- (56, 0) -- (75, -5) -- (100, 0) -- (105, 5) -- (144, 10) -- (173, 5) -- (222, 0) -- (263, -5) -- (304, 0) -- (315, 5) -- (330, 10) -- (335, 15) ;
\end{tikzpicture}
\caption{The reduced random walk from Definitions \ref{fig:red_walk}.}
\label{fig:red_walk}
\end{figure}

The black graph represents the walk $\omega$. 
The thick blue line represents $\red\left(\omega,L\right)$, 
but note that the vertical coordinate is multiplied by $L$ when it is overlayed 
on the graph of $\omega$, and that the horizontal coordinate is adjusted to the 
rate of $\omega$ - so one should look at whether the blue graph goes up 
or down, rather than at the $t$ for which it does. 
Here are the first values of the reduced walk in this illustration:
\[
\red\left(\omega,L\right) = 
\left(
	-1,+1,-1,+1,+1,+1,-1,-1,-1,+1,+1,\dots
\right).
\]

The rectangles represent the intervals 
$\left\{ i_{\omega,L}\left(j\right),\dots,i_{\omega,L}\left(j+1\right)-1\right\}$
, similar to the ones from Definition \ref{defn:intervals}. 
Their horizontal range is exactly this interval, 
and their vertical range is the window in which the random walk moves until 
it exits the interval by reaching 
$i_{\omega,L}\left(j+1\right) = \next\left(\omega,L,i_{\omega,L}\left(j\right)\right)$
; namely, 
$\left[
X^{\omega}_{i_{\omega,L}\left(j\right)} -L,
X^{\omega}_{i_{\omega,L}\left(j\right)} +L
\right]$.

\begin{defn}
Let $n$ be a positive integer and let $p,\bar{p}\in\left[0,1\right]$.
Let $Y\sim\mathrm{Bin}\left(n,p\right)$ be a random variable
\footnote{Recall that this means that $\Pr\left(Y=k\right)={\binom{n}{k}} p^{k} \left(1-p\right)^{n-k}$ for $k\in\left\{0,1,\dots,n\right\}$.}
. Denote:
\[
\pbin\left(n,p,\bar{p}\right)=\Pr\left(Y>\bar{p} n\right).
\]
\end{defn}
\begin{lem}
\label{lem:chernoff_pbin_bound}
For $n,p,\bar{p}$ as in the former definition and $0<t<1$, if $t\bar{p} \ge p$
 then:
\[
\pbin\left(n,p,\bar{p}\right)<\exp\left(-n\left(1-t\right)^{2}\bar{p}^{2}\right).
\]
\end{lem}
\begin{proof}
The proof is a straightforward application of Chernoff bound.
Denote by $Y$ a binomial random variable with $n$ independent trials with probability $p$ of success.
Let $Y_{1},\dots,Y_{n}$ be Bernoulli random variables with probability $p$, such that $Y=Y_{1}+\cdots+Y_{n}$.
Define $X_{i}=Y_{i}-p$ for each $i\in\left\{1,\dots,n\right\}$ and $S=X_{1}+\cdots+X_{n}$. The sequence $X_{i}$ is of independent random variables,
with each $\mathbb{E}\left[X_{i}\right]=0$, and no two values of $X_{i}$ are more than $1$ apart.
By the variation on Chernoff bound found at Theorem A1.18 in \cite{alon_spencer}, the statements in the last sentence regarding the $X_{i}$-s imply that for any $a>0$ 
\[
\Pr\left(S>a\right) < e^{-2a^{2}/n},
\]
and therefore:
\begin{multline*}
	\begin{aligned}
	\pbin\left(n,p,\bar{p}\right) 
	& = \Pr\left(Y > n\bar{p}\right) \\
	& = \Pr\left(Y_{1}+\cdots+Y_{n}-pn > n\bar{p}-pn\right) \\
	& = \Pr\left(\right(Y_{1}-p\left)+\cdots+\right(Y_{n}-p\left) > \left(\bar{p}-p\right)n\right) \\
	& = \Pr\left(X_{1}+\cdots+X_{n} > \left(\bar{p}-p\right)n\right) \\
	& = \Pr\left(S > \left(\bar{p}-p\right)n\right) \\
	& < \exp\left(-2\left(\left(\bar{p}-p\right)n\right)^{2}/n\right) \\
	& = \exp\left(-2n\left(\bar{p}-p\right)^{2}\right) \\
	& < \exp\left(-2n\left(1-t\right)^{2}\bar{p}^{2}\right).
	\end{aligned}
\end{multline*}
\end{proof}

\subsection{Abnormally Long Intervals}

\begin{rem}
\label{rem:implicit_parameters}
By $M^{\mathrm{upper}}_{m}$ and $R^{\mathrm{upper}}_{m}$ where $m\in\mathbb{N}$, we denote parameters for whom values will be specified only at Definition \ref{defn:sinwalks}.
The sets $\badb_{m}^{\mathrm{upper}}\left(\omega\right)$ that we are about to define are defined using them, but since it would be cumbersome to denote the set by 
$\badb_{m}^{\mathrm{upper}}\left(\omega,M^{\mathrm{upper}}_{m}\right)$, we will leave this parameter out of the notation.

Later in this section we will define other sets with similar notation, which will depend on more parameters such as 
$M^{\mathrm{lower}}_{m}$, $R^{\mathrm{lower}}_{m}$ and $\beta_{m}$ where $m\in\mathbb{N}$. In order to avoid notations such as 
$\badb_{m}\left(\omega,M^{\mathrm{upper}}_{m},M^{\mathrm{lower}}_{m},R^{\mathrm{upper}}_{m},R^{\mathrm{lower}}_{m},\beta_{m}\right)$, 
we will similarly use these parameters implicitly in our definitions of 
``good'' and ``bad''
sets later in this paper 
(specifically, Definitions 
\ref{def:badbupper}, 
\ref{def:badbredupper}, 
\ref{def:badblower}, 
\ref{def:badbredlower}, 
\ref{def:badb_length}, 
\ref{def:badb_redlength}, 
\ref{def:badblocal}, 
\ref{def:badb_all}
).
\end{rem}

\begin{defn}
\label{def:badb_asterix_at_levels}

For $m,m'\in\left\{0,\dots,k\right\}$ such that $m'\le m$, $m\ge1$ and $\omega\in\walks$, we will use the following notations from here on:
\[
\badb_{m,m'}^{\mathrm{*}}\left(\omega\right)\coloneqq\bigcup_{j\in\badb_{m}^{\mathrm{*}}\left(\omega\right)}I_{m,j}^{m'}\left(\omega\right),
\]
\[
\goodb_{m}^{\mathrm{*}}\left(\omega\right) \coloneqq
\left\{ 0,1,\dots,\len_{m}\left(\omega\right)-1\right\} 
\backslash\badb_{m}^{\mathrm{*}}\left(\omega\right),
\]
\[
\goodb_{m,m'}^{\mathrm{*}}\left(\omega\right) \coloneqq\bigcup_{j\in\goodb_{m}^{\mathrm{*}}\left(\omega\right)}I_{m,j}^{m'}\left(\omega\right).\\
\]
Here $\mathrm{*}$ either stands for one of the notations 
$\mathrm{upper}, \mathrm{redUpper}$, $\mathrm{lower}, \mathrm{redLower}$, $\mathrm{length}, \mathrm{redLength}$ and $\mathrm{local}$
that we will use in Definitions 
\ref{def:badbupper}, 
\ref{def:badbredupper},
\ref{def:badblower}, 
\ref{def:badbredlower}, 
\ref{def:badb_length}, 
\ref{def:badb_redlength} and
\ref{def:badblocal}, 
where $\badb_{m}^{\mathrm{*}}\left(\omega\right)$ will be defined, 
or for the appropriate notations for Definition \ref{def:badb_all}, where $\badb_{m}\left(\omega\right)$ will be defined.

\end{defn}

Let us give two illustrations of this definition, with $L=L_{1}\cdots L_{m}$:

\begin{figure}[h]
\tikzsetnextfilename{illustration2}
\begin{tikzpicture}[y=.1cm, x=.0007cm,font=\sffamily]
 	\draw[->] (0,0) -- coordinate (x axis mid) (10500,0);
	\draw[->] (0,-45) -- coordinate (second x axis mid) (10500,-45);
    \draw[<->] (0,-35) -- coordinate (y axis mid) (0,15);
    \foreach \y in {-3,-2,...,1}
        \draw (1pt,\y*10) -- (-3pt,\y*10) 
     	    node[anchor=east] {\y $L$}; 
	\node[right=4cm] at (x axis mid) {$t$};
	\node[right=4cm] at (second x axis mid) {$t$};
	\node[above=2.5cm] at (y axis mid) {$X^{\omega}_{t}$};
	\draw plot[mark=., mark options={fill=white}] 
		file {div_soft_badb_upper.data};
    \foreach \y in {-30,-20,...,10}
    	\draw[-,dotted] (0,\y) -- (10500,\y);
     \draw[pattern=north east lines]  (6,10) rectangle (852,-10)  (852,0) rectangle (1372,-20)  (2288,-20) rectangle (3570,0)  (5760,-10) rectangle (10000,10);
	\draw  (1372,-30) rectangle (2288,-10)  (3570,10) rectangle (4190,-10)   (4190,-20) rectangle (5760,0);
	\draw[pattern=north east lines]  (6,-45) rectangle (852,-40)  (852,-45) rectangle (1372,-40)  (2288,-45) rectangle (3570,-40)  (5760,-45) rectangle (10000,-40);
	\draw  (1372,-45) rectangle (2288,-40)  (3570,-45) rectangle (4190,-40)  (4190,-45) rectangle (5760,-40);
	\draw[fill=black] (0,-45) circle (0.03cm);
\end{tikzpicture}
\caption{Illustration of Definition \ref{def:badb_asterix_at_levels}. 
}
\label{fig:badb_asterix_level}
\end{figure}
In Figure \ref{fig:badb_asterix_level}, $m=1$, $\badb_{m}^{\mathrm{*}}\left(\omega\right) = \left\{0,1,3,6\right\}$, and ``bad'' intervals are represented 
by rectangles shaded by diagonals, while ``good'' intervals are 
represented by non-shaded rectangles. The lower part represents $\badb_{m,0}^{\mathrm{*}}\left(\omega\right)$ and $\goodb_{m,0}^{\mathrm{*}}\left(\omega\right)$.

\begin{figure}[h]
\tikzsetnextfilename{illustration2.5}
\begin{tikzpicture}[y=.5cm, x=.0012cm,font=\sffamily]
 	\draw[->] (0,0) -- coordinate (x axis mid) (5760+100,0);
	\node[right=4cm] at (x axis mid) {$t$};
    \node at (-1000,0.5) {$\badb_{1,0}^{\mathrm{*}}\left(\omega\right)$};
	\draw[pattern=north east lines]  (234,0) rectangle (432,1)  (1000,0) rectangle (1372,1)  (1372,0) rectangle (1456,1)  (1818,0) rectangle (2288,1)  (2500,0) rectangle (2801,1)  (2801,0) rectangle (3141,1)  (3570,0) rectangle (3900,1)  (3900,0) rectangle (4190,1);
	\draw  (0,0) rectangle (234,1)  (432,0) rectangle (852,1)  (852,0) rectangle (1000,1)  (1456,0) rectangle (1818,1)  (2288,0) rectangle (2500,1)  (3141,0) rectangle (3570,1)  (4190,0) rectangle (5760,1);
	\node at (-1000,1.5) {$\badb_{2,0}^{\mathrm{*}}\left(\omega\right)$};
	\draw[pattern=north east lines]  (0,1) rectangle (432,2)  (1818,1) rectangle (2801,2)  (3570,1) rectangle (5760,2);
	\draw  (432,1) rectangle (1372,2)  (1818,1) rectangle (2801,2)  (2801,1) rectangle (3570,2);
	\node at (-1000,2.5) {$\badb_{3,0}^{\mathrm{*}}\left(\omega\right)$};
	\draw[pattern=north east lines]  (0,2) rectangle (1372,3)  (1372,2) rectangle (2801,3);
	\draw  (2801,2) rectangle (5760,3);
	\draw[fill=black] (0,0) circle (0.03cm);
\end{tikzpicture}
\caption{Another illustration of Definition \ref{def:badb_asterix_at_levels}. }
\label{fig:badb_asterix_multilevel}
\end{figure}

In Figure \ref{fig:badb_asterix_multilevel}, we see the hierarchical structure of the $\badb_{m,0}^{\mathrm{*}}\left(\omega\right)$-s, that follows 
from the hierarchical structure of the reduced walk.

\begin{defn}
\label{def:badbupper}
For $m\in\left\{1,\dots,k\right\}$ and $\omega\in\walks$:
\[
\badb_{m}^{\mathrm{upper}}\left(\omega\right) \coloneqq
\left\{ j\in\left\{ 0,1,\dots,\len_{m}\left(\omega\right)-1\right\} :
	\left|I_{m,j}^{0}\left(\omega\right)\right|>M_{m}^{\mathrm{upper}}\right\}. 
\]
Moreover, $\goodb_{m}^{\mathrm{upper}}\left(\omega\right),\badb_{m,m'}^{\mathrm{upper}}\left(\omega\right)$ and $\goodb_{m,m'}^{\mathrm{upper}}\left(\omega\right)$ 
are defined as in Definition \ref{def:badb_asterix_at_levels}.
\end{defn}
\begin{defn}
\label{def:badbredupper}
For $m\in\left\{1,\dots,k\right\}$ and $\omega\in\walks$:
\[
\badb_{m}^{\mathrm{redUpper}}\left(\omega\right) \coloneqq
\left\{ j\in\left\{ 0,1,\dots,\len_{m}\left(\omega\right)-1\right\} :
\left|I_{m,j}\left(\omega\right)\right|>R_{m}^{\mathrm{upper}}\right\}.
\]
Moreover, $\goodb_{m}^{\mathrm{redUpper}}\left(\omega\right),\badb_{m,m'}^{\mathrm{redUpper}}\left(\omega\right)$ and $\goodb_{m,m'}^{\mathrm{redUpper}}\left(\omega\right)$ 
are defined as in Definition \ref{def:badb_asterix_at_levels}.
\end{defn}

The Definitions \ref{def:badbupper} and \ref{def:badbredupper} are used in order to throw away all the $j$-s for which 
$\left|I_{m,j}^{0}\left(\omega\right)\right|$ and $\left|I_{m,j}\left(\omega\right)\right|$ are abnormally long, or in other words in which the function $\next$ returns a value that 
is much larger than expected for $i_{\omega,L_{1}\cdots L_{m}}\left(j\right)$. 
Recall that $I_{m,j}\left(\omega\right) = I_{m,j}^{m-1}\left(\omega\right)$ is generally much smaller than $I_{m,j}^{0}\left(\omega\right)$, and thus we will have that $R_{m}^{\mathrm{upper}}$ is generally much smaller than $M_{m}^{\mathrm{upper}}$. 
It would be useful to assume that this doesn't happen, and therefore we will try to control the number of such $j$-s, and then find some way 
to ``throw them away''.

In fact, we already ``throw away'' some of the record $x=\left(\left(\omega\left(t\right),\sigma\left(X^{\omega}_{t}\right)\right)\right)_{t=0}^{N-1}$
from Theorem \ref{thm:main_theorem}: 
We don't use the entire set $\left\{0\le t < N : x\left(t\right) \neq x'\left(t\right) \right\}$ whose size is bounded by $\delta N$. 
Thus, it will not be terribly different to add to this set another $\delta N$ new  $t$-s, in which the random walk was inside abnormally long  $I_{m,j}^{0}\left(\omega\right)$. 
This is the goal of Definition \ref{def:badb_asterix_at_levels}: identify inconvenient parts of the walk, 
bound them inside a set of size $\approx\delta N$ times, and treat them as if they were part of the adversarially inserted errors that we avoid.

\subsubsection{Probability of spending too much time in $\protect\badb_{m,0}^{\mathrm{upper}}\left(\omega\right)$ 
and $\protect\badb_{m,m-1}^{\mathrm{redUpper}}\left(\omega\right)$}

\begin{lem}
\label{lem:bound_on_badb_uppers_volume}

Let $m\in\left\{1,\dots,k\right\}$, and let $0<\alpha<1$. Then:
\begin{enumerate}
\item
\begin{multline*}
\Pr_{\omega}\left(\left|\badb_{m,m-1}^{\mathrm{redUpper}}\left(\omega\right)\right|>
		\alpha\len_{m-1}\left(\omega\right)  		\text{ and }  \omega\in\regwalksall\right)\\
<
\pbin\left(\frac{4}{R_{m}^{\mathrm{upper}}}N_{m-1},\Pr_{\omega}\left(\next\left(\omega,2L_{m},0\right)>\frac{R_{m}^{\mathrm{upper}}}{4}\right),\frac{\alpha}{2}\right),
\end{multline*}
\item
\begin{multline*}
\Pr_{\omega}\left(\left|\badb_{m,0}^{\mathrm{upper}}\left(\omega\right)\right|>
	\alpha\len\left(\omega\right)  	\text{ and }  \omega\in\regwalksall\right)\\
<\pbin\left(\frac{4}{M_{m}^{\mathrm{upper}}}N_{0},\Pr_{\omega}\left(\next\left(\omega,2L_{1}\cdots L_{m},0\right)>\frac{M_{m}^{\mathrm{upper}}}{4}\right),\frac{\alpha}{2}\right).
\end{multline*}
\end{enumerate}
\end{lem}
\begin{proof}The two properties are very similar, and their proofs are almost identical. Therefore for the sake of simplicity we will show only (1).
First divide $\left\{0,1,\dots,\len_{m-1}\left(\omega\right)\right\}$ to intervals of length $\frac{1}{4}R^{\mathrm{upper}}_{m}$.
There are $\frac{4\len_{m-1}\left(\omega\right)}{R^{\mathrm{upper}}_{m}}$ such intervals, and for each $j\in\badb_{m}^{\mathrm{redUpper}}\left(\omega\right)$ 
at least half 
of the indices $i\in I_{m,j}\left(\omega\right)$ are contained in such an interval that is a subset of $I_{m,j}\left(\omega\right)$ (as opposed to partly 
contained in $I_{m,j}\left(\omega\right)$ and partly in $I_{m,j-1}\left(\omega\right)$ or $I_{m,j+1}\left(\omega\right)$).

If an interval of length $\frac{1}{4}R^{\mathrm{upper}}_{m}$ is contained in some $I_{m,j}\left(\omega\right)$, it implies that the position $X^{\omega}_{t}$ in this interval 
cannot change by more than $2L_{m}$, and the probability that this happens is $p=\Pr_{\omega}\left(\next\left(\omega,2L_{m},0\right)>\frac{R_{m}^{\mathrm{upper}}}{4}\right)$. 
Also note that the second property occurs independently for different intervals.
To conclude, we get that 
\[
\left|\badb_{m,m-1}^{\mathrm{redUpper}}\left(\omega\right)\right|>
		\alpha\len_{m-1}\left(\omega\right)  		\text{ and }  \omega\in\regwalksall
\]
 implies that out of $\frac{4\len_{m-1}\left(\omega\right)}{R^{\mathrm{upper}}_{m}} > \frac{4 N_{m-1} }{R^{\mathrm{upper}}_{m}}$ intervals a fraction of over $\frac{1}{2}\alpha$ 
satisfy a property that holds with probability $p$, and the events of this happening are independent. This implies the part (1) of the lemma. The following illustration shows the ingredients of the proof:

\begin{figure}[h]
\tikzsetnextfilename{illustration3}
\begin{tikzpicture}[y=.1cm, x=.0007cm,font=\sffamily]
	\draw (0,-40) -- coordinate (x axis mid) (10000,-40);
    \draw (0,-40) -- coordinate (y axis mid) (0,15);
    \foreach \x in {0,2,...,10}
        \draw (\x*1000,-112pt) -- (\x*1000,-116pt)
            node[anchor=north] {$\frac{\x}{4}R^{\mathrm{upper}}_{m}$};
    \foreach \x in {1,3,...,9}
        \draw (\x*1000,-116pt) -- (\x*1000,-112pt)
            node[anchor=south] {$\frac{\x}{4}R^{\mathrm{upper}}_{m}$};
    \foreach \y in {-4,-3,-2,...,1}     		\draw (1pt,10*\y) -- (-3pt,10*\y) 
        node[anchor=east] {$\y L_{m}$}; 
    \node[below=0.8cm] at (x axis mid) {$t$};
    \node[left=0.8cm] at (y axis mid) {$X^{\omega}_{t}$};
    \draw plot[mark=., mark options={fill=white}] 
        file {div_soft_badb_upper.data};
    \foreach \y in {-30,-20,...,10}
    	\draw[-,dotted] (0,\y) -- (10000,\y);
	\foreach \x in {0,1000,...,10000}
    	\draw[green,-] (\x,-40) -- (\x,15);
    \draw[->,blue,mark=x]  (6,0)  --  (852,-10)  --  (1372,-20)  --  (2288,-10)  --  (3570,0)  --  (4190,-10)  --  (5760,0)  ;
    \draw[red] (6,10) rectangle (852,-10)   (852,0) rectangle (1372,-20)   (1372,-30) rectangle (2288,-10)   (2288,-20) rectangle (3570,0)   (3570,10) rectangle (4190,-10)   (4190,-20) rectangle (5760,0) ;
    \draw[red] (10000,10) -- (5760,10) -- (5760,-10) -- (10000,-10) ;
\end{tikzpicture}
\caption{Idea of the proof of Lemma \ref{lem:bound_on_badb_uppers_volume}}
\label{fig:bound_on_badb_uppers_volume_proof}
\end{figure}
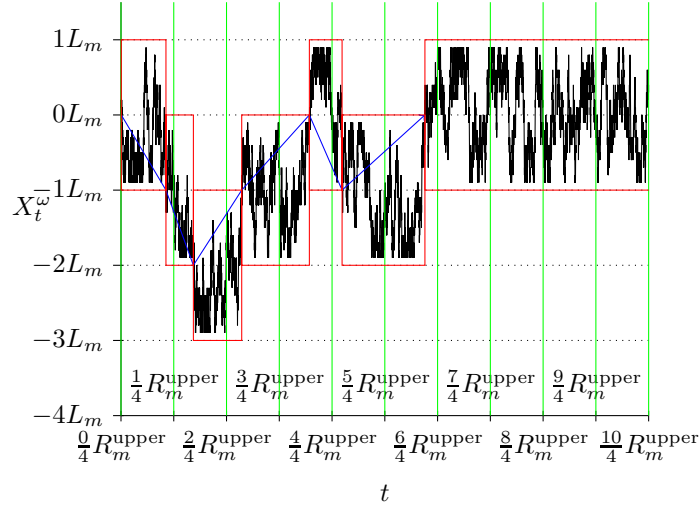

In Figure \ref{fig:bound_on_badb_uppers_volume_proof}, $m=1$, and we see that  $\left\{0,1,2,3,4,5\right\}\subseteq\goodb_{m}^{\mathrm{redUpper}}\left(\omega\right)$ while $6\in\badb_{m}^{\mathrm{redUpper}}\left(\omega\right)$. 
Indeed, over half of $I_{m,6}\left(\omega\right)$ is covered by the intervals of length $\frac{1}{4}R^{\mathrm{upper}}_{m}$ (separated by green vertical lines in the illustration) with range of $X^{\omega}_{t}$ smaller than $2L_{m}$.
The proof of part (2) is similar and left to the reader.

\end{proof}

\begin{lem}
\label{lem:prob lower bound on next of walk}
For $L,N\in\mathbb{N}$ such that $N>80L^{2}$ and $L\ge1000$ it
holds that:
\[
\Pr_{\omega}\left(\next\left(\omega,L,0\right)>N\right)<\exp\left(-\frac{1}{27}\frac{N}{L^{2}}\right).
\]
\end{lem}
\begin{proof}
It is equivalent to say that $\max_{0\le i<N}\left|X_{i}^{\omega}\right|<L$ and 
that $\next\left(\omega,L,0\right)>N$. Therefore:
\begin{multline*}
\Pr\left(\omega:\next\left(\omega,L,0\right)>N\right)
  =  \Pr_{\omega}\left(\max_{0\le i\le N}\left|X_{i}^{\omega}\right|<L\right)\\
\begin{aligned}
 & \le  \Pr_{\omega}\left(
 -L\le X_{i}^{\omega}\le L \text{ for $0\leq i<N$}\right)\\
 & \le  \Pr_{\omega}\left(
 -L\le X_{4 k L^2+j}^{\omega}\le L
 \text{ for $0\leq j<4L^{2}$, $0\le k <\left\lfloor \frac{N}{4L^{2}}\right\rfloor -1$}\right)\\
 & \le  \Pr_{\omega}\left(
 -2L\le X_{4 k L^2+j}^{\omega}-X_{4 k L^2}^{\omega}\le2L 
 \text{ for $0\leq j<4L^{2}$, $0\le k <\left\lfloor \frac{N}{4L^{2}}\right\rfloor -1$}\right).
 \end{aligned}
\end{multline*}
Now, since the random walk $\omega$ is memoryless: 
\begin{multline}
\label{eq:next_bound1}
	\Pr\left(\omega:\next\left(\omega,L,0\right)>N\right) < \\
	\begin{aligned}
		& \Pr\left(
				\omega:
				-2L\le X_{j}^{\omega}\le2L
				\text{ for $j\in\left\{ 0,1,\dots,4L^{2}-1\right\}$}
			\right)^{\left\lfloor \frac{N}{4L^{2}}\right\rfloor -1}.
	\end{aligned}
\end{multline}
Thus, by simply using the definition of the probability measure on random walks,
\begin{multline}
\label{eq:next_bound2}
	\Pr_{\omega}\left(\forall j\in\left\{ 0,1,\dots,4L^{2}-1\right\}: \,
	-2L\le X_{j}^{\omega}\le2L\right)\\
	\begin{aligned}
	& <\Pr_{\omega}\left(\,-2L\le X_{4L^{2}}^{\omega}\le2L\right)\\
	& =\Pr\left( 2L^{2}-L\le\mathrm{Bin}\left(4L^{2},\frac{1}{2}\right)\le 2L^{2}+L \right),
	\end{aligned}
\end{multline}
where $\mathrm{Bin}\left(n,p\right)$ denotes a binomial random variable.
As $L$ goes to infinity, this probability goes to the probability that a normal random variable is at most $1$ standard deviation away from its expectation. 
Since we assumed that $L\ge1000$, we may find by calculation some $C_0>\eqref{eq:next_bound2}$, and get:

\begin{multline*}
\eqref{eq:next_bound1} <
\exp\left(\ln{C_0}
\left(\left\lfloor \frac{N}{4L^{2}}\right\rfloor -1\right)\right) \\<
\exp\left(0.9\ln{C_0} \cdot 
\frac{N}{4L^{2}}\right)
 =
 \exp\left(0.225\ln{C_0} \frac{N}{L^{2}}\right).
\end{multline*}
Now, since the aforementioned calculation gives that 
\[
C_0 = 
\exp\left(-\frac{4\cdot\frac{1}{27}}{0.9}\right)
>\eqref{eq:next_bound2}
,\]
and also
\[
0.225 \ln C_0 =
-\frac{1}{27}
,\]
we get the required bound for the conclusion of the lemma.

\end{proof}

In the rest of this paper, we will apply Lemma \ref{lem:prob lower bound on next of walk} repeatedly, with $N=\frac{R^{\mathrm{upper}}_{m}}{4}$, or with 
$N=\frac{M^{\mathrm{upper}}_{m}}{4}$. For the sake of efficiency, we define the following notation for its premise in these cases. This condition, and the other conditions that will follow, restrict the parameters of our construction. At the end of Section \ref{section_sinwalks}, we will give explicitly parameters that satisfy our restrictions.
\begin{crit}
\label{crit:ratios_criterion}
We say that Condition \ref{crit:ratios_criterion} holds if 
for any $m\in\left\{1,\dots,k\right\}$ it holds that:
\begin{align}
& L_{m}\ge 1000, & \qquad \\
& R_{m}^{\mathrm{upper}}>1280 {L_{m}}^{2}, &\\
& M_{m}^{\mathrm{upper}}>1280\left(L_{1}\cdots L_{m}\right)^{2}. &
\end{align}
\end{crit}

By the last lemma, it holds that:
\begin{cor}
\label{cor:bound_prob_badb_lower}
If Condition \ref{crit:ratios_criterion} holds, 
then for any $m\in\left\{1,\dots,k\right\}$:
\begin{equation}\label{eq:cor:bound_prob_badb_lower eq1}
 \Pr_{\omega}\left(\next\left(\omega,2L_{m},0\right)>\frac{R_{m}^{\mathrm{upper}}}{4}\right)
  <\exp\left(-\frac{1}{432}\frac{R_{m}^{\mathrm{upper}}}{L_{m}^{2}}\right),
\end{equation}
\begin{equation}\label{eq:cor:bound_prob_badb_lower eq2}
 \Pr_{\omega}\left(\next\left(\omega,2L_{1}\cdots L_{m},0\right)>\frac{M_{m}^{\mathrm{upper}}}{4}\right)
  <\exp\left(-\frac{1}{432}\frac{M_{m}^{\mathrm{upper}}}{\left(L_{1}\cdots L_{m}\right)^{2}}\right),
\end{equation}
\begin{equation}\label{eq:cor:bound_prob_badb_lower eq3}
\Pr_{\omega}\left(0\in\badb^{\mathrm{redUpper}}_{m}\left(\omega\right)\right)
<\exp\left(-\frac{1}{27}\frac{R_{m}^{\mathrm{upper}}}{L_{m}^{2}}\right),
\end{equation}
\begin{equation}\label{eq:cor:bound_prob_badb_lower eq4}
\Pr_{\omega}\left(0\in\badb^{\mathrm{upper}}_{m}\left(\omega\right)\right)
<\exp\left(-\frac{1}{27}\frac{M_{m}^{\mathrm{upper}}}{\left(L_{1} \cdots L_{m}\right)^{2}}\right).
\end{equation}
\end{cor}

\begin{lem}
\label{lem:bad_upper_m_bound}
Let $m\in\left\{1,\dots,k\right\}$, and let $0<\alpha<1$. Suppose Condition \ref{crit:ratios_criterion} holds.
\begin{enumerate}
\item
If $\alpha>4\exp\left(-\frac{1}{432}\frac{R_{m}^{\mathrm{upper}}}{L_{m}^{2}}\right)$, 
then:
\[
\Pr_{\omega}\left(\left|\badb_{m,m-1}^{\mathrm{redUpper}}\left(\omega\right)\right|>
\alpha\len_{m-1}\left(\omega\right)
\text{ and } \omega\in\regwalksall\right)
<\exp\left(-\frac{1}{4R_{m}^{\mathrm{upper}}}N_{m-1} {\alpha}^{2}\right).
\]
\item
If $\alpha>4\exp\left(-\frac{1}{432}\frac{M_{m}^{\mathrm{upper}}}{\left(L_{1}\cdots L_{m}\right)^{2}}\right)$, 
then:
\[
\Pr_{\omega}\left(\left|\badb_{m,0}^{\mathrm{upper}}\left(\omega\right)\right|>
\alpha\len_{0}\left(\omega\right)
\text{ and } \omega\in\regwalksall\right)
<\exp\left(-\frac{1}{4M_{m}^{\mathrm{upper}}}N_{0} {\alpha}^{2}\right).
\]
\end{enumerate}
\end{lem}
\begin{proof}
We will show only the proof of (1), since the proof of (2) is very similar. 
Denote $p=\exp\left(-\frac{1}{432}\frac{R_{m}^{\mathrm{upper}}}{L_{m}^{2}}\right)$. 
By Lemma \ref{lem:bound_on_badb_uppers_volume}, Corollary \ref{cor:bound_prob_badb_lower} and Lemma \ref{lem:chernoff_pbin_bound}:
											  
\begin{multline*}
\Pr_{\omega}\left(\left|\badb_{m,m-1}^{\mathrm{redUpper}}\left(\omega\right)\right|>
	\alpha\len_{m-1}\left(\omega\right)  		\text{ and }  \omega\in\regwalksall\right)\\
\begin{aligned}
  <& \pbin\left(\frac{4N_{m-1}}{R_{m}^{\mathrm{upper}}},\Pr_{\omega}\left(\next\left(\omega,2L_{m},0\right)>\frac{R_{m}^{\mathrm{upper}}}{4}\right),\frac{\alpha}{2}\right)
\\<& \pbin\left(\frac{4N_{m-1}}{R_{m}^{\mathrm{upper}}},p,\frac{\alpha}{2}\right)
\\<& \exp\left(-\frac{N_{m-1}}{4} \cdot \frac{4}{R_{m}^{\mathrm{upper}}} \cdot \left(\frac{\alpha}{2}\right)^{2}\right)
\\<& \exp\left(-\frac{1}{4R_{m}^{\mathrm{upper}}}N_{m-1} {\alpha}^{2}\right).
\end{aligned}
\end{multline*}

\end{proof}

\subsection{Abnormally Short Intervals}

The following two definitions have implicit parameters analogous to $M_{m}^{\mathrm{upper}}$ in the definition of $\badb_{m}^{\mathrm{upper}}\left(\omega\right)$, 
which will be chosen later, as was discussed in Remark \ref{rem:implicit_parameters}.
 \begin{defn}
\label{def:badblower}
For $m\in\left\{1,\dots,k\right\}$ and $\omega\in\walks$:
\[
\badb_{m}^{\mathrm{lower}}\left(\omega\right)\coloneqq
\left\{ j\in\left\{ 0,1,\dots,\len_{m}\left(\omega\right)-1\right\} :
	\left|I_{m,j}^{0}\left(\omega\right)\right|<M_{m}^{\mathrm{lower}}\right\}.
\]
Moreover, $\goodb_{m}^{\mathrm{lower}}\left(\omega\right),\badb_{m,m'}^{\mathrm{lower}}\left(\omega\right)$ and $\goodb_{m,m'}^{\mathrm{lower}}\left(\omega\right)$ 
are defined as in Definition \ref{def:badb_asterix_at_levels}.
\end{defn}
\begin{defn}
\label{def:badbredlower}
For $m\in\left\{1,\dots,k\right\}$ and $\omega\in\walks$:
\[
\badb_{m}^{\mathrm{redLower}}\left(\omega\right)\coloneqq
\left\{ j\in\left\{ 0,1,\dots,\len_{m}\left(\omega\right)-1\right\} 
	:\left|I_{m,j}\left(\omega\right)\right|<R_{m}^{\mathrm{lower}}\right\}.
\]
Moreover, $\goodb_{m}^{\mathrm{redLower}}\left(\omega\right),\badb_{m,m'}^{\mathrm{redLower}}\left(\omega\right)$ and $\goodb_{m,m'}^{\mathrm{redLower}}\left(\omega\right)$ 
are defined as in Definition \ref{def:badb_asterix_at_levels}.
\end{defn}

\subsubsection{Probability of spending too much time in $\protect\badb_{m,0}^{\mathrm{lower}}\left(\omega\right)$
 and $\protect\badb_{m,m-1}^{\mathrm{redLower}}\left(\omega\right)$}
 
\begin{lem}
\label{lem:bad_lower_m_bound_without_asymptotics}
Let $m\in\left\{1,\dots,k\right\}$, and let $0<\alpha<1$. Then:
\begin{enumerate}
\item
\begin{multline*}
\Pr_{\omega}\left(\left|\badb_{m,m-1}^{\mathrm{redLower}}\left(\omega\right)\right|>
\alpha\len_{m-1}\left(\omega\right) \text{ and } \omega\in\regwalksall\right)\\<\pbin\left(N_{m},\Pr_{\omega}\left(0\in\badb_{m}^{\mathrm{redLower}}\left(\omega\right)\right),\alpha\right),
\end{multline*}
\item
\begin{multline*}
\Pr_{\omega}\left(\left|\badb_{m,0}^{\mathrm{lower}}\left(\omega\right)\right|>
\alpha\len\left(\omega\right) \text{ and } \omega\in\regwalksall\right)\\<\pbin\left(N_{m},\Pr_{\omega}\left(0\in\badb_{m}^{\mathrm{lower}}\left(\omega\right)\right),\alpha\right).
\end{multline*}
\end{enumerate}
\end{lem}
\begin{proof}
The proofs of (1) and (2) are nearly identical, so we will show just the proof of (1). 
First of all note that 
\begin{multline*}
\Pr_{\omega}\left(\left|\badb_{m,m-1}^{\mathrm{redLower}}\left(\omega\right)\right|>
\alpha\len_{m-1}\left(\omega\right) \text{ and } \omega\in\regwalksall\right)\\
<\Pr_{\omega}\left(\left|\badb_{m,m-1}^{\mathrm{redLower}}\left(\omega\right)\right|>
\alpha\len_{m-1}\left(\omega\right) \text{ and } \omega\in\regwalks_{\ge}^{m}\right)\\
<\Pr_{\omega}\left(\left|\badb_{m,m-1}^{\mathrm{redLower}}\left(\omega\right)\right|>
\alpha\len_{m-1}\left(\omega\right)\Biggm|\regwalks_{\ge}^{m}\right).\\
\end{multline*}
Also, note that 
\[\left|\badb_{m,m-1}^{\mathrm{redLower}}\left(\omega\right)\right| < R^{\mathrm{lower}}_{m} \left|\badb_{m}^{\mathrm{redLower}}\left(\omega\right)\right|\] and
\[\left|\goodb_{m,m-1}^{\mathrm{redLower}}\left(\omega\right)\right| \ge R^{\mathrm{lower}}_{m} \left|\goodb_{m}^{\mathrm{redLower}}\left(\omega\right)\right|.\]
Therefore:
\begin{multline*}
\Pr_{\omega}\left(\left|\badb_{m,m-1}^{\mathrm{redLower}}\left(\omega\right)\right|>
\alpha\len_{m-1}\left(\omega\right)\Biggm|\regwalks_{\ge}^{m}\right)\\ 
\begin{aligned} =&
\Pr_{\omega}\left(\left(1-\alpha\right)\left|\badb_{m,m-1}^{\mathrm{redLower}}\left(\omega\right)\right|>
\alpha\left|\goodb_{m,m-1}^{\mathrm{redLower}}\left(\omega\right)\right|\Biggm|\regwalks_{\ge}^{m}\right)\\ <&
\Pr_{\omega}\left(\left(1-\alpha\right) R^{\mathrm{lower}}_{m} \left|\badb_{m}^{\mathrm{redLower}}\left(\omega\right)\right|>
\alpha R^{\mathrm{lower}}_{m} \left|\goodb_{m}^{\mathrm{redLower}}\left(\omega\right)\right|\Biggm|\regwalks_{\ge}^{m}\right)\\ =&
\Pr_{\omega}\left( \left|\badb_{m}^{\mathrm{redLower}}\left(\omega\right)\right|>
\alpha \len_{m}\left(\omega\right) \Biggm|\regwalks_{\ge}^{m}\right)\\ <&
\pbin\left(N_{m},\Pr_{\omega}\left(0\in\badb_{m}^{\mathrm{redLower}}\left(\omega\right)\right),\alpha\right).
\end{aligned}
\end{multline*}
\end{proof}
\begin{lem}
\label{lem:probablistic lower bound on next}
Let $L,N$ be positive integers. It holds that:
\[
\Pr_{\omega}\left(\next\left(\omega,L,0\right)<N\right)<N\exp\left(-\frac{L^{2}}{N}\right).
\]
Therefore, for $m\in\left\{1,\dots,k\right\}$ it holds that:
\[
\Pr_{\omega}\left(0\in\badb_{m}^{\mathrm{lower}}\left(\omega\right)\right)<M_{m}^{\mathrm{lower}}\exp\left(-\frac{\left(L_{1}\cdots L_{m}\right)^{2}}{M_{m}^{\mathrm{lower}}}\right),
\]
\[
\Pr_{\omega}\left(0\in\badb_{m}^{\mathrm{redLower}}\left(\omega\right)\right)<R_{m}^{\mathrm{lower}}\exp\left(-\frac{L_{m}^{2}}{R_{m}^{\mathrm{lower}}}\right).
\]
\end{lem}
\begin{proof}
Since $\next\left(\omega,L,0\right)<N$ is equivalent to $\max_{0\le i<N}\left|X_{i}^{\omega}\right|\ge L$,
we have:
\begin{eqnarray*}
\Pr_{\omega}\left(\max_{j\in\left\{ 0,\dots,N-1\right\} }\left|X_{j}^{\omega}\right|\ge L\right) & = & \Pr\left(\bigcup_{j=0}^{N-1}\left\{ \omega:\left|X_{j}^{\omega}\right|\ge L\right\} \right)\\
 & = & \Pr\left(\bigcup_{j=1}^{N-1}\left\{ \omega:\left|X_{j}^{\omega}\right|\ge L\right\} \right)\\
 & < & \sum_{j=1}^{N-1}\Pr\left(\left\{ \omega:\left|X_{j}^{\omega}\right|\ge L\right\} \right),
\end{eqnarray*}
and due to Hoeffding inequality \cite{Hoeffding}:
\begin{eqnarray*}
\Pr_{\omega}\left(\max_{j\in\left\{ 0,\dots,N-1\right\} }\left|X_{j}^{\omega}\right|\ge L\right) & < & \sum_{j=1}^{N-1}\exp\left(-\frac{L^{2}}{j}\right)\\
 & < & N\exp\left(-\frac{L^{2}}{N}\right).
\end{eqnarray*}
\end{proof}

\begin{lem}
\label{lem:bad_lower_m_bound}
Let $m\in\left\{1,\dots,k\right\}$ and let $0<\alpha<1$. 
\begin{enumerate}
\item
If $\alpha>2 R_{m}^{\mathrm{lower}}\exp\left(-\frac{L_{m}^{2}}{R_{m}^{\mathrm{lower}}}\right)$ then:
\[
\Pr_{\omega}\left(\left|\badb_{m,m-1}^{\mathrm{redLower}}\left(\omega\right)\right|>
\alpha\len_{m-1}\left(\omega\right)\text{ and } \omega\in\regwalksall\right)
<\exp\left(-\frac{1}{4}N_{m}\alpha^2\right). \]
\item
If $\alpha>2 M_{m}^{\mathrm{lower}}\exp\left(-\frac{\left(L_{1}\cdots L_{m}\right)^{2}}{M_{m}^{\mathrm{lower}}}\right)$ then:
\[
\Pr_{\omega}\left(\left|\badb_{m,0}^{\mathrm{lower}}\left(\omega\right)\right|>
\alpha\len\left(\omega\right)\text{ and } \omega\in\regwalksall\right)
<\exp\left(-\frac{1}{4}N_{m}\alpha^2\right). \]
\end{enumerate}
\end{lem}
\begin{proof}
By combining Lemma \ref{lem:bad_lower_m_bound_without_asymptotics}, Lemma \ref{lem:probablistic lower bound on next} and Lemma \ref{lem:chernoff_pbin_bound}, we get:
\begin{multline*}
\Pr_{\omega}\left(\left|\badb_{m,m-1}^{\mathrm{redLower}}\left(\omega\right)\right|>
\alpha\len_{m-1}\left(\omega\right) \text{ and } \omega\in\regwalksall\right)\\
\begin{aligned}
  <& \pbin\left(N_{m},\Pr_{\omega}\left(0\in\badb_{m}^{\mathrm{redLower}}\left(\omega\right)\right),\alpha\right)
\\<& \pbin\left(N_{m},R_{m}^{\mathrm{lower}}\exp\left(-\frac{L_{m}^{2}}{R_{m}^{\mathrm{lower}}}\right),\alpha\right)
\\<& \exp\left(-\frac{1}{4}N_{m}\alpha^2\right).
\end{aligned}
\end{multline*}

\end{proof}

\subsection{Intervals of Abnormal Length}

\begin{defn}
\label{def:badb_length}
For $m,m'\in\left\{0,\dots,k\right\}$ such that $m'\le m$, $m\ge1$ and $\omega\in\walks$:
\[
\badb_{m}^{\mathrm{length}}\left(\omega\right)\coloneqq\badb_{m}^{\mathrm{lower}}\left(\omega\right)\cup\badb_{m}^{\mathrm{upper}}\left(\omega\right).
\]
Moreover, $\goodb_{m}^{\mathrm{length}}\left(\omega\right),\badb_{m,m'}^{\mathrm{length}}\left(\omega\right)$ and $\goodb_{m,m'}^{\mathrm{length}}\left(\omega\right)$ 
are defined as in Definition \ref{def:badb_asterix_at_levels}.
\end{defn}
\begin{defn}
\label{def:badb_redlength}
For $m,m'\in\left\{0,\dots,k\right\}$ such that $m'\le m$, $m\ge1$ and $\omega\in\walks$:
\[
\badb_{m}^{\mathrm{redLength}}\left(\omega\right)\coloneqq\badb_{m}^{\mathrm{redLower}}\left(\omega\right)\cup\badb_{m}^{\mathrm{redUpper}}\left(\omega\right).
\]
Moreover, $\goodb_{m}^{\mathrm{redLength}}\left(\omega\right),\badb_{m,m'}^{\mathrm{redLength}}\left(\omega\right)$ and $\goodb_{m,m'}^{\mathrm{redLength}}\left(\omega\right)$ 
are defined as in Definition \ref{def:badb_asterix_at_levels}.
\end{defn}

\subsubsection{Probability of spending too much time in $\protect\badb_{m,0}^{\mathrm{length}}\left(\omega\right)$ 
and $\protect\badb_{m,m-1}^{\mathrm{redLength}}\left(\omega\right)$}
Now we wish to combine Lemma \ref{lem:bad_lower_m_bound} and Lemma \ref{lem:bad_upper_m_bound} to a single statement. To avoid repeating their assumptions too many times, we will define two new Conditions on the value of $0<\alpha<1$:
\begin{crit}
\label{crit:alpha_bound_for_redLength_bound}
Condition \ref{crit:alpha_bound_for_redLength_bound} holds for $m\in\left\{1,\dots,k\right\}$ and $0<\alpha<1$ if:
\begin{align} 
&\alpha>4\exp\left(-\frac{1}{432}\frac{R_{m}^{\mathrm{upper}}}{L_{m}^{2}}\right) && \quad\text{This is used in Lemma \ref{lem:bad_upper_m_bound} (1)}  \\
&\alpha>2 R_{m}^{\mathrm{lower}}\exp\left(-\frac{L_{m}^{2}}{R_{m}^{\mathrm{lower}}}\right) && \quad\text{This is used in Lemma \ref{lem:bad_lower_m_bound} (1)} \\
&\alpha>\frac{8 L_{m}R^{\mathrm{upper}}_{m}}{R^{\mathrm{lower}}_{m}} \exp\left(-\frac{ \beta_{m} R^{\mathrm{lower}}_{m} }{L_{m}}\right)
&& \quad\text{This will be used  in Lemma \ref{lem:bad_redlength_local_m_m-1_bounds}.}\end{align}
\end{crit}
\begin{crit}
\label{crit:alpha_bound_for_length_bound}
Condition \ref{crit:alpha_bound_for_length_bound} holds for $m\in\left\{1,\dots,k\right\}$ and $0<\alpha<1$ if:
\begin{align}
&  \alpha>4\exp\left(-\frac{1}{432}\frac{M_{m}^{\mathrm{upper}}}{\left(L_{1}\cdots L_{m}\right)^{2}}\right) 
&& \qquad\text{This is used in Lemma \ref{lem:bad_upper_m_bound} (2)}\\
&  \alpha>2 M_{m}^{\mathrm{lower}}\exp\left(-\frac{\left(L_{1}\cdots L_{m}\right)^{2}}{M_{m}^{\mathrm{lower}}}\right) 
&& \qquad\text{This is used in Lemma \ref{lem:bad_lower_m_bound} (2)}
\end{align}
\end{crit}

\begin{lem}
\label{lem:bad_lengths_m_bounds}
Let $m\in\left\{1,\dots,k\right\}$ and let $0<\alpha<1$. Suppose Condition \ref{crit:ratios_criterion} holds.
\begin{enumerate}
\item If $\alpha$ and $m$ satisfy Condition \ref{crit:alpha_bound_for_redLength_bound} then:
\begin{multline*}
\Pr_{\omega}\left(\left|\badb_{m,m-1}^{\mathrm{redLength}}\left(\omega\right)\right|>
2\alpha\len_{m-1}\left(\omega\right)\text{ and } \omega\in\regwalksall\right)
\\ <
\exp\left(-\frac{1}{4R_{m}^{\mathrm{upper}}}N_{m-1} {\alpha}^{2}\right) +
\exp\left(-\frac{1}{4}N_{m}\alpha^2\right).
\end{multline*}
\item If $\alpha$ and $m$ satisfy Condition \ref{crit:alpha_bound_for_length_bound} then:
\begin{multline*}
\Pr_{\omega}\left(\left|\badb_{m,0}^{\mathrm{length}}\left(\omega\right)\right|>
2\alpha\len\left(\omega\right)
\text{ and } \omega\in\regwalksall\right) \\<
\exp\left(-\frac{1}{4M_{m}^{\mathrm{upper}}}N_{0} {\alpha}^{2}\right) + \exp\left(-\frac{1}{4}N_{m}\alpha^2\right).
\end{multline*}
\end{enumerate}
\end{lem}
\begin{proof}
The proof is a straightforward combination of Lemma \ref{lem:bad_lower_m_bound} and Lemma \ref{lem:bad_upper_m_bound}.

By Lemma \ref{lem:bad_upper_m_bound}, since Condition \ref{crit:ratios_criterion} holds, and since by Condition \ref{crit:alpha_bound_for_redLength_bound} \[\alpha>4\exp\left(-\frac{1}{432}\frac{R_{m}^{\mathrm{upper}}}{L_{m}^{2}}\right),\]
we conclude
\begin{multline*}
\Pr_{\omega}\left(\left|\badb_{m,m-1}^{\mathrm{redUpper}}\left(\omega\right)\right|>
\alpha\len_{m-1}\left(\omega\right)
\text{ and } \omega\in\regwalksall\right)
\\<\exp\left(-\frac{1}{4R_{m}^{\mathrm{upper}}}N_{m-1} {\alpha}^{2}\right).
\end{multline*}
By \ref{lem:bad_lower_m_bound}, since $\alpha>2 R_{m}^{\mathrm{lower}}\exp\left(-\frac{L_{m}^{2}}{R_{m}^{\mathrm{lower}}}\right)$ 
by Condition \ref{crit:alpha_bound_for_redLength_bound}, then:
\[
\Pr_{\omega}\left(\left|\badb_{m,m-1}^{\mathrm{redLower}}\left(\omega\right)\right|>
\alpha\len_{m-1}\left(\omega\right)
\text{ and } \omega\in\regwalksall\right)
<\exp\left(-\frac{1}{4}N_{m}\alpha^2\right),
\]
and together, since 
$\left|\badb_{m,m-1}^{\mathrm{redLower}}\left(\omega\right)\right|+\left|\badb_{m,m-1}^{\mathrm{redUpper}}\left(\omega\right)\right|\ge
\left|\badb_{m,m-1}^{\mathrm{redLength}}\left(\omega\right)\right|$, we get:
\begin{multline*}
\Pr_{\omega}\left(\left|\badb_{m,m-1}^{\mathrm{redLength}}\left(\omega\right)\right|>
2\alpha\len_{m-1}\left(\omega\right)
\text{ and } \omega\in\regwalksall
\right) \\ 
\begin{aligned} \le&
\Pr_{\omega}\left(\left|\badb_{m,m-1}^{\mathrm{redLower}}\left(\omega\right)\right|+
\left|\badb_{m,m-1}^{\mathrm{redUpper}}\left(\omega\right)\right|>
2\alpha\len_{m-1}\left(\omega\right)
\text{ and } \omega\in\regwalksall
\right) \\ \le&
\Pr_{\omega}\left(\left|\badb_{m,m-1}^{\mathrm{redLower}}\left(\omega\right)\right|>
\alpha\len_{m-1}\left(\omega\right)
\text{ and } \omega\in\regwalksall
\right) +\\&+
\Pr_{\omega}\left(\left|\badb_{m,m-1}^{\mathrm{redUpper}}\left(\omega\right)\right|>
\alpha\len_{m-1}\left(\omega\right)
\text{ and } \omega\in\regwalksall
\right) \\ <&
\exp\left(-\frac{1}{4R_{m}^{\mathrm{upper}}}N_{m-1} {\alpha}^{2}\right) +
\exp\left(-\frac{1}{4}N_{m}\alpha^2\right).
\end{aligned}
\end{multline*}

The proof of (2) follows in a similar way from the same lemmas and the relevant condition.

\end{proof}

\subsection{Intervals Where the Random Walk Induces Atypical Local Time Measure}
\label{subsection:local_time}
Consider the probability distribution of the position $X^{\omega}_{t}$. Informally speaking, In a typical random walk, this position is unlikely to spend a lot of time 
in specific places or specific intervals. If this typical behaviour doesn't occur, then the adversary can use this to confuse us. Thus, we wish to remove from the walk this times when this atypical behavior takes place.
\begin{defn}
For $m\in\left\{1,\dots,k\right\},\omega\in\walks$ and for $j\in\left\{ 0,1,\dots,\len_{m}\left(\omega\right)-1\right\} $
we define:
\[
\overline{MLT}_{m,j}\left(\omega\right)=\max_{x\in\mathbb{Z}}\left|\left\{ t\in I_{m,j}\left(\omega\right):X_{t}^{\red\left(\omega,L_{1}\cdots L_{m-1}\right)}=x\right\} \right|.
\]
\end{defn}

In the following definition, $\beta_{m}$ is a parameter 
that will be chosen later, as we noted in Remark \ref{rem:implicit_parameters}.
\begin{defn}
\label{def:badblocal}
For $m\in\left\{1,\dots,k\right\}$ and $\omega\in\walks$:
\[
\badb_{m}^{\mathrm{local}}\left(\omega\right)\coloneqq\left\{ j\in\left\{ 0,1,\dots,\len_{m}\left(\omega\right)-1\right\} :\overline{MLT}_{m,j}\left(\omega\right)>\frac{\beta_{m}}{2}\left|I_{m,j}\left(\omega\right)\right|\right\}.
\]
Moreover, $\goodb_{m}^{\mathrm{local}}\left(\omega\right),\badb_{m,m'}^{\mathrm{local}}\left(\omega\right)$ and $\goodb_{m,m'}^{\mathrm{local}}\left(\omega\right)$ 
are defined as in Definition \ref{def:badb_asterix_at_levels}.
\end{defn}

\subsubsection{Probability of spending too much time in $\protect\badb_{m,m-1}^{\mathrm{local}}\left(\omega\right)$}
\begin{lem}
\label{cor:bound_prob_badb_local}
\label{lem:0inbadlocal_prob}
Assume Condition \ref{crit:ratios_criterion}, and let $m\in\left\{1,\dots,k\right\}$. Then:
\[
\Pr_{\omega}\left(0\in\badb_{m}^{\mathrm{local}}\left(\omega\right)\cap\goodb_{m}^{\mathrm{redLength}}\left(\omega\right)\right)
< 4 L_{m} \exp\left(-\frac{ \beta_{m} R^{\mathrm{lower}}_{m} }{L_{m}}\right).
\]
\end{lem}
\begin{proof}
The solution to the classical gambler's ruin problem states that a simple random walk starting at $0$ that is finished on the first time that the walker reaches some points 
$a>0$ or $-b<0$, has a chance $\frac{b}{a+b}$ to reach $a$ first, and a chance $\frac{a}{a+b}$ to reach $b$ first (cf. \cite{feller}, chapter XIV, section 2, p.344).

Consider a random walk starting at $0$. What is the probability that the walk returns to $0$ before reaching $L$ or $-L$? After the first step, 
without loss of generality, the walker stands at $1$. Then the probability that the walker reaches $0$ before $L$ is known to be $\frac{L-1}{L}=1-\frac{1}{L}$. 

Now, what is the probability that the walker visits some $k$ at least $n$ times before reaching $\pm L$? Assume $k\ge0$. The chance we reach $k$ before $-L$ is 
$\frac{L}{L+k}$. Then we step either to $k-1$ or $k+1$, and we ask whether we reach $-L$ or $L$ respectively before returning to $k$. 
The probability that we don't is:
\[
\frac{1}{2} \left(1-\frac{1}{L-k}\right) + \frac{1}{2} \left(1-\frac{1}{L+k}\right)
=
1-\frac{L}{L^{2}-k^{2}}.
\]
Thus the probability that we visit $k$ exactly $n$ times before reaching $\pm L$ is:
\[
\frac{L}{L+k} \left( 1-\frac{L}{L^{2}-k^{2}} \right)^{n-1}
< \left( 1-\frac{1}{L\left(1-\frac{k^{2}}{L^{2}}\right)} \right)^{n-1},
\]
and this is an upper bound on the probability that the walk visits $k$ (or $-k$) more than $n$ times before it reaches $\pm L$. 
We can use a union bound, and get that the probability that a random walks visits any location more than $n$ times before it reaches $\pm L$ is bounded by 
:
\begin{equation}
\label{eq:nice_local_time_bound}
\sum_{k=-L+1}^{L-1} \left( 1-\frac{1}{L\left(1-\frac{k^{2}}{L^{2}}\right)} \right)^{n-1} <
2L \left(1-\frac{1}{L}\right)^{n-1}.
\end{equation}

Now, if $0\in\badb_{m}^{\mathrm{local}}\left(\omega\right)\cap\goodb_{m}^{\mathrm{redLength}}\left(\omega\right)$ then the walk $\red\left(\omega,L_{1}\cdots L_{m-1}\right)$ 
visits some place more than $\beta_{m}\left|I_{m,0}\left(\omega\right)\right|$, which is larger than $\beta_{m} R^{\mathrm{lower}}_{m}$, before it reaches $\pm L_{m}$.
Thus, \eqref{eq:nice_local_time_bound} and Condition \ref{crit:ratios_criterion} yield:
\begin{align*}
\Pr_{\omega}\left(0\in\badb_{m}^{\mathrm{local}}\left(\omega\right)\cap\goodb_{m}^{\mathrm{redLength}}\left(\omega\right)\right)
&< 2 L_{m} \left(1-\frac{1}{L_{m}}\right)^{ \beta_{m} R^{\mathrm{lower}}_{m} -1}
\\&< 4 L_{m} \exp\left(-\frac{ \beta_{m} R^{\mathrm{lower}}_{m} }{L_{m}}\right).
\end{align*}
\end{proof}

\begin{lem}
\label{lem:bad_redlength_local_m_m-1_bounds}
Let $m\in\left\{1,\dots,k\right\}$ and $0<\alpha<1$. 
Suppose Condition \ref{crit:ratios_criterion} holds, and that $\frac{\alpha}{8}$ and $m$ satisfy Condition \ref{crit:alpha_bound_for_redLength_bound}. Then:

\begin{multline*}
Pr_{\omega}\left(\left|\badb^{\mathrm{local}}_{m,m-1}\left(\omega\right) \cup \badb_{m,m-1}^{\mathrm{redLength}}\left(\omega\right)\right|>
4\alpha\len_{m-1}\left(\omega\right)\text{ and } \omega\in\regwalksall\right)
\\ < 
2\exp\left(-\frac{1}{64R_{m}^{\mathrm{upper}}}N_{m-1} {\alpha}^{2} \right) +
4\exp\left(-\frac{1}{256}N_{m}\left( \alpha \frac{R^{\mathrm{lower}}_{m}}{R^{\mathrm{upper}}_{m}} \right)^{2}\right).
\end{multline*}
\end{lem}
\begin{proof}
First we wish to give a probabilistic bound on the size of  $\badb^{\mathrm{local}}_{m,m-1}\left(\omega\right) \cap \goodb^{\mathrm{redLength}}_{m,m-1}\left(\omega\right)$.
For the sake of brevity, denote:
\[
A_{m} = \badb^{\mathrm{local}}_{m}\left(\omega\right) \cap \goodb^{\mathrm{redLength}}_{m}\left(\omega\right),
\]
\[
A_{m-1} = \badb^{\mathrm{local}}_{m,m-1}\left(\omega\right) \cap \goodb^{\mathrm{redLength}}_{m,m-1}\left(\omega\right).
\]
Then 
\begin{equation}\label{eq:A_m A_m-1}
\left|A_{m-1}\right|<R^{\mathrm{upper}}_{m} \left|A_{m}\right|
\end{equation} 
and 
\begin{equation}\label{eq:good m m-1}
\left| \goodb^{\mathrm{redLength}}_{m,m-1}\left(\omega\right) \right|>R^{\mathrm{lower}}_{m} \left| \goodb^{\mathrm{redLength}}_{m}\left(\omega\right) \right|.
\end{equation}
Therefore:
\begin{multline}\label{eq:Pr estimate1}
\Pr_{\omega}\left( 
	\left|A_{m-1}\right|>\alpha 
	\len_{m-1}\left(\omega\right)
\text{ and } \omega\in\regwalksall \right)
\\
\begin{aligned}
<&
\Pr_{\omega}\left( 
	\left|A_{m-1}\right|>2\alpha 
	\left| \goodb^{\mathrm{redLength}}_{m,m-1}\left(\omega\right) \right|
\text{ and } \omega\in\regwalksall \right)
+\\&+
\Pr_{\omega}\left( 
	\left| \goodb^{\mathrm{redLength}}_{m,m-1}\left(\omega\right) \right|
	<\frac{1}{2}\len_{m-1}\left(\omega\right)
\text{ and } \omega\in\regwalksall \right).
\end{aligned}
\end{multline}
Applying \eqref{eq:good m m-1} and \eqref{eq:A_m A_m-1}, 
and the fact that $\regwalks_{\ge}^{m}\supset \regwalksall$,
it follows that
\begin{align*}
\eqref{eq:Pr estimate1}<&
\Pr_{\omega}\left( 
	\left|A_{m}\right|>2\alpha \frac{R^{\mathrm{lower}}_{m}}{R^{\mathrm{upper}}_{m}}
	\left| \goodb^{\mathrm{redLength}}_{m}\left(\omega\right) \right|
\text{ and } \omega\in\regwalks_{\ge}^{m} \right)
+\\&+
\Pr_{\omega}\left( 
	\left| \goodb^{\mathrm{redLength}}_{m,m-1}\left(\omega\right) \right|
	<\frac{1}{2}\len_{m-1}\left(\omega\right)
\text{ and } \omega\in\regwalksall \right)
\end{align*}
hence
\begin{equation}\label{eq: three terms}
\begin{aligned}
\eqref{eq:Pr estimate1}<&
\Pr_{\omega}\left( 
	\left|A_{m}\right|>\alpha \frac{R^{\mathrm{lower}}_{m}}{R^{\mathrm{upper}}_{m}}
	\len_{m}\left(\omega\right)
\text{ and } \omega\in\regwalks_{\ge}^{m} \right)
+\\&+
\Pr_{\omega}\left( 
	\left| \goodb^{\mathrm{redLength}}_{m}\left(\omega\right) \right|
	<\frac{1}{2}\len_{m}\left(\omega\right)
\text{ and } \omega\in\regwalks_{\ge}^{m} \right)
+\\&+
\Pr_{\omega}\left( 
	\left| \goodb^{\mathrm{redLength}}_{m,m-1}\left(\omega\right) \right|
	<\frac{1}{2}\len_{m-1}\left(\omega\right)
\text{ and } \omega\in\regwalksall \right).
\end{aligned}
\end{equation}

Now, by Lemma \ref{cor:bound_prob_badb_local} and our assumptions, the first term in \eqref{eq: three terms} may be bounded as follows:
\begin{multline*}
\Pr_{\omega}\left( 
	\left|A_{m}\right|>\alpha \frac{R^{\mathrm{lower}}_{m}}{R^{\mathrm{upper}}_{m}}
	\len_{m}\left(\omega\right)
\middle| \omega\in\regwalks_{\ge}^{m} \right) \\
\begin{aligned} <&
\Pr_{\omega}\left( 
	\left|A_{m}\right|>\alpha \frac{R^{\mathrm{lower}}_{m}}{R^{\mathrm{upper}}_{m}}
	\len_{m}\left(\omega\right)
\middle| \len_{m}\left(\omega\right)=N_{m} \right) \\ =&
\pbin\left(N_{m},
\Pr_{\omega}\left(0\in\badb_{m}^{\mathrm{local}}\left(\omega\right)\cap\goodb_{m}^{\mathrm{redLength}}\left(\omega\right)\right),
\alpha \frac{R^{\mathrm{lower}}_{m}}{R^{\mathrm{upper}}_{m}}\right) \\ <&
\pbin\left(N_{m},
4 L_{m} \exp\left(-\frac{ \beta_{m} R^{\mathrm{lower}}_{m} }{L_{m}}\right),
\alpha \frac{R^{\mathrm{lower}}_{m}}{R^{\mathrm{upper}}_{m}}\right).
\end{aligned}
\end{multline*}
Condition \ref{crit:alpha_bound_for_redLength_bound} implies that 
$
\frac{\alpha}{2} \frac{R^{\mathrm{lower}}_{m}}{R^{\mathrm{upper}}_{m}}
>
4 L_{m} \exp\left(-\frac{ \beta_{m} R^{\mathrm{lower}}_{m} }{L_{m}}\right)
$, 
so by Lemma \ref{lem:chernoff_pbin_bound}:
\begin{multline}
\label{eq:term1 bound}
\Pr_{\omega}\left( 
	\left|A_{m}\right|>\alpha \frac{R^{\mathrm{lower}}_{m}}{R^{\mathrm{upper}}_{m}}
	\len_{m}\left(\omega\right)
\text{ and } \omega\in\regwalks_{\ge}^{m} \right) \\ <
\pbin\left(N_{m},
\frac{\alpha}{2} \frac{R^{\mathrm{lower}}_{m}}{R^{\mathrm{upper}}_{m}}, 
\alpha \frac{R^{\mathrm{lower}}_{m}}{R^{\mathrm{upper}}_{m}}\right) <
\exp\left(-\frac{1}{16}N_{m}\left( \alpha \frac{R^{\mathrm{lower}}_{m}}{R^{\mathrm{upper}}_{m}} \right)^{2}\right) \\<
\exp\left(-\frac{1}{256}N_{m}\left( \alpha \frac{R^{\mathrm{lower}}_{m}}{R^{\mathrm{upper}}_{m}} \right)^{2}\right).
\end{multline}

The second term in \eqref{eq: three terms} may be bounded by:
\begin{multline}
\label{eq:term2 bound_part1}
\Pr_{\omega}\left( 
	\left| \goodb^{\mathrm{redLength}}_{m}\left(\omega\right) \right|
	<\frac{1}{2}\len_{m}\left(\omega\right)
\text{ and } \omega\in\regwalks_{\ge}^{m} \right) \\ 
\begin{aligned}
<&
\Pr_{\omega}\left( 
	\left| \badb^{\mathrm{redLength}}_{m}\left(\omega\right) \right|
	>\frac{1}{4}\len_{m}\left(\omega\right)
\middle| \omega\in\regwalks_{\ge}^{m} \right) \\ <&
\pbin\left(N_{m},\Pr_{\omega}\left(0\in\badb^{\mathrm{redLength}}_{m}\left(\omega\right)\right),\frac{1}{4}\right).
\end{aligned}
\end{multline}

Since Condition \ref{crit:ratios_criterion} implies that $R_{m}^{\mathrm{upper}}>80L_{m}^{2}$, $L_{m}\ge1000$, we see that
$\exp\left(-\frac{1}{27}\frac{R_{m}^{\mathrm{upper}}}{L_{m}^{2}}\right)<\frac{1}{16}$ and 
$R_{m}^{\mathrm{lower}}\exp\left(-\frac{L_{m}^{2}}{R_{m}^{\mathrm{lower}}}\right)<\frac{1}{16}$, 
and thus by Lemmas \ref{lem:probablistic lower bound on next} and \ref{cor:bound_prob_badb_lower}, 
$\Pr_{\omega}\left(0\in\badb^{\mathrm{redLength}}_{m}\left(\omega\right)\right)<\frac{1}{8}$. Therefore we get by Lemma \ref{lem:chernoff_pbin_bound}:
\begin{equation}
\label{eq:term2 bound}
\eqref{eq:term2 bound_part1}
< \exp\left(-\frac{1}{256}N_{m}\right)
< \exp\left(-\frac{1}{256}N_{m}\left( \alpha \frac{R^{\mathrm{lower}}_{m}}{R^{\mathrm{upper}}_{m}} \right)^{2}\right).
\end{equation}

For the third term, we use Lemma \ref{lem:bad_lengths_m_bounds}, whose premise holds since we assumed 
Condition \ref{crit:ratios_criterion} holds, and that some value $\frac{1}{2}>\frac{\alpha}{8}$ and $m$ satisfy Condition \ref{crit:alpha_bound_for_redLength_bound}:
\begin{multline}
\label{eq:term3 bound}
\Pr_{\omega}\left( 
	\left| \goodb^{\mathrm{redLength}}_{m,m-1}\left(\omega\right) \right|
	<\frac{1}{2}\len_{m-1}\left(\omega\right)
\text{ and } \omega\in\regwalksall \right) \\
\begin{aligned} = &
\Pr_{\omega}\left( 
	\left| \badb^{\mathrm{redLength}}_{m,m-1}\left(\omega\right) \right|
	>\frac{1}{2}\len_{m-1}\left(\omega\right)
\text{ and } \omega\in\regwalksall \right) \\ < &
\exp\left(-\frac{1}{64R_{m}^{\mathrm{upper}}}N_{m-1} \right) +
\exp\left(-\frac{1}{64}N_{m}\right) \\ < &
\exp\left(-\frac{1}{64R_{m}^{\mathrm{upper}}}N_{m-1} {\alpha}^{2} \right) +
\exp\left(-\frac{1}{256}N_{m}\left( \alpha \frac{R^{\mathrm{lower}}_{m}}{R^{\mathrm{upper}}_{m}} \right)^{2}\right).
\end{aligned}
\end{multline}

Now we may combine \eqref{eq: three terms},\eqref{eq:term1 bound},\eqref{eq:term2 bound} and \eqref{eq:term3 bound} to get:
\begin{equation}
\eqref{eq:Pr estimate1} < 
\exp\left(-\frac{1}{64R_{m}^{\mathrm{upper}}}N_{m-1} {\alpha}^{2} \right) +
3\exp\left(-\frac{1}{256}N_{m}\left( \alpha \frac{R^{\mathrm{lower}}_{m}}{R^{\mathrm{upper}}_{m}} \right)^{2}\right),
\end{equation}
and this yields the conclusion of this lemma.

\end{proof}

\subsection{Probability of spending too much time in $\badb_{m,0}\left(\omega\right)$}

\begin{defn}
\label{def:badb_all}
For $m,m'\in\left\{0,\dots,k\right\}$ such that $m'\le m$, $m\ge1$ and $\omega\in\walks$:
\[
\badb_{m}\left(\omega\right)\coloneqq\badb_{m}^{\mathrm{length}}\left(\omega\right)\cup\badb_{m}^{\mathrm{redLength}}\left(\omega\right)\cup\badb_{m}^{\mathrm{local}}\left(\omega\right).
\]
Moreover, $\goodb_{m}\left(\omega\right),\badb_{m,m'}\left(\omega\right)$ and $\goodb_{m,m'}\left(\omega\right)$ 
are defined as in Definition \ref{def:badb_asterix_at_levels}.
\end{defn}

Up until now, most of the work we did was regarding $\badb_{m}\left(\omega\right)$. In this subsection, we extend this to $\badb_{m,0}\left(\omega\right)$.
\begin{lem}
\label{lem:sum_of_delta_gt_epsilon}
\label{lem:sum_of_nexts}
Let $N$ be a positive integer, $\delta N$ also a positive integer for $0<\delta<\frac{1}{2}$. Consider a sequence of i.i.d random variables $X_{1},\dots,X_{N}$ with the same probability law as $\next\left(\omega,L,0\right)$. It holds that:
\[
\Pr_{\omega}\left(\sum_{i=1}^{\delta N}X_{i}> \sqrt{\delta} \sum_{i=1}^{N}X_{i}\right)
< \frac{5}{N}.
\]

\end{lem}
\begin{proof}

First we wish to calculate the variance and expectation of $X_{i}$. The calculation of both is by the method used in 
\cite{feller}, chapter XIV, section 3. Define random variables:
\[
\tau_{n} \coloneqq \min\left\{ t \ge 0 : n+X_{t}^{\omega} \in \left\{\pm L \right\} \right\}
\]
By our definition of $\next$, $X_{i}$ is distributed with the same probability law as $\tau_{0}$. Note that:
\begin{equation}
\label{eq:boundary value of stopping time}
\tau_{L} = \tau_{-L} = 0
\end{equation}
By the law of total expectation:
\begin{align*}
\mathbb{E}\left[\tau_{n}\right] & = 
\frac{1}{2} \mathbb{E}\left[\tau_{n-1}+1\right] + 
\frac{1}{2} \mathbb{E}\left[\tau_{n+1}+1\right]
\\ & =
\frac{1}{2} \mathbb{E}\left[\tau_{n-1}\right] + 
\frac{1}{2} \mathbb{E}\left[\tau_{n+1}\right] + 1,
\end{align*}
and by \eqref{eq:boundary value of stopping time} the boundary conditions for this difference equation is 
$\mathbb{E}\left[\tau_{L}\right] = \mathbb{E}\left[\tau_{-L}\right] = 0$. 
The solution for this equation is uniquely $\mathbb{E}\left[\tau_{n}\right] = L^2-n^2$ for $-L\le n \le L$.

Similarly, by the law of total expectation, for $-L<n<L$:
\begin{align*}
\mathbb{E}\left[\tau_{n}^2\right] & = 
\frac{1}{2} \mathbb{E}\left[\left(\tau_{n-1}+1\right)^2\right] + 
\frac{1}{2} \mathbb{E}\left[\left(\tau_{n+1}+1\right)^2\right]
\\ & =
\frac{1}{2} \mathbb{E}\left[\tau_{n-1}^2+2\tau_{n-1}+1\right] + 
\frac{1}{2} \mathbb{E}\left[\tau_{n+1}^2+2\tau_{n+1}+1\right]
\\ & = 
\frac{1}{2} \mathbb{E}\left[\tau_{n-1}^2\right] + 
\frac{1}{2} \mathbb{E}\left[\tau_{n+1}^2\right] + 
2L^{2} - \left(n-1\right)^2 - \left(n+1\right) + 1,
\end{align*}
and again by \eqref{eq:boundary value of stopping time},
$\mathbb{E}\left[\tau_{L}^2\right] = \mathbb{E}\left[\tau_{-L}^2\right] = 0$. The solution now is:
\[
\mathbb{E}\left[\tau_{n}^2\right] = \frac{1}{3}\left(L^2-n^2\right)\left(5L^2-n^2-2\right).
\]
Therefore:
\[
\Var\left(\tau_{n}\right) = 
\mathbb{E}\left[\tau_{n}^2\right] - \mathbb{E}\left[\tau_{n}\right]^2 = 
\frac{2}{3}\left(L^2-n^2\right)\left(L^2+n^2-1\right),
\]
and in particular, $\mathbb{E}\left[X_{i}\right] = L^2$ and $\Var\left(X_{i}\right) = \frac{2}{3}L^2\left(L^2-1\right) < \frac{2}{3} L^4$.

Define a new random variable $Y = \left(1-\sqrt{\delta}\right)\sum_{i=1}^{\delta N}X_{i} - \sqrt{\delta} \sum_{i=\delta N+1}^{N}X_{i}$. Note that:
\begin{equation}
\label{eq:chebyshev_term1}
\Pr_{\omega}\left(\sum_{i=1}^{\delta N}X_{i}> \sqrt{\delta} \sum_{i=1}^{N}X_{i}\right) =
\Pr_{\omega}\left(Y>0\right) <
\Pr_{\omega}\left(\left|Y-\mathbb{E}\left[Y\right]\right|>\mathbb{E}\left[Y\right]\right).
\end{equation}
Also note that:
\begin{equation}
\label{eq:chebyshev_term_exp}
\left(\mathbb{E}\left[Y\right]\right)^2 = 
\left(- \left(1-\sqrt{\delta}\right)\delta N L^{2} +\sqrt{\delta}\left(1-\delta\right) N L^2 \right)^2 =
\left( \sqrt{\delta}-\delta  \right)^2  N^{2} L^{4},
\end{equation}
and since the $X_{i}$ are independent random variables:
\begin{multline}
\label{eq:chebyshev_term_var}
\Var\left(Y\right) = 
\left(1-\sqrt{\delta}\right)^2\sum_{i=1}^{\delta N}\Var\left(X_{i}\left)+\delta \sum_{i=\delta N+1}^{N}\Var\right(X_{i}\right) \\ =
\left(\left(1-\sqrt{\delta}\right)^2 \delta + \delta \left(1-\delta\right)\right) N \Var\left(X_{1}\right) <
\frac{4}{3} \delta\left(1 - \sqrt{\delta} \right) N L^4.
\end{multline}
By applying Chebyshev's inequality, we get from \eqref{eq:chebyshev_term1}, \eqref{eq:chebyshev_term_exp} and \eqref{eq:chebyshev_term_var}:
\begin{equation*}
\eqref{eq:chebyshev_term1}
 < \frac{\Var\left(Y\right)}{\left(\mathbb{E}\left[Y\right]\right)^2}
 < \frac{4}{3\left(1-\sqrt{\delta}\right)N}.
\end{equation*}
Since we assumed $0<\delta<\frac{1}{2}$, this yield the inequality in the lemma.

\end{proof}

Now we can define a new condition, that controls the probability bad regions of the walks being visited for too much time.
\begin{crit}
\label{crit:bound_on_probs}
Condition \ref{crit:bound_on_probs} holds for $m\in\left\{1,\dots,k-1\right\}$ and $0<\alpha<1$ if:
\begin{enumerate}
\item $\alpha>2\exp\left(-\frac{1}{64R_{m}^{\mathrm{upper}}}N_{m-1} {\alpha}^{4} \right)$
\item $\alpha>5\exp\left(-\frac{1}{256}N_{m}\left( \alpha^{2} \frac{R^{\mathrm{lower}}_{m}}{R^{\mathrm{upper}}_{m}} \right)^{2}\right)$
\item $\alpha>\exp\left(-\frac{1}{4M_{m}^{\mathrm{upper}}}N_{0} {\alpha}^{2}\right)$
\item $\alpha>\frac{5}{N_{m-1}}$.
\end{enumerate}
\end{crit}

Later we will use Lemma \ref{lem:condition_on_alpha} to see that a sufficient condition for Condition \ref{crit:bound_on_probs} is that 
\begin{multline*}
\alpha >
\max{\Bigg(}
	4 \left(\frac{N_{m-1}}{R^{\mathrm{upper}}_{m}}\right)^{-\frac{1}{4}},
	2\sqrt{5} \left(N_{m}\left(\frac{R^{\mathrm{lower}}_{m}}{R^{\mathrm{upper}}_{m}}\right)^{2}\right)^{-\frac{1}{8}},\\
	\frac{1}{\sqrt{2}} \left(\frac{N_{0}}{M^{\mathrm{upper}}_{m}}\right)^{-\frac{1}{4}},
	5\left(N_{m-1}\right)^{-1}
{\Bigg)}.
\end{multline*}

\begin{lem}
\label{lem:bad_redlength_local_m_bounds}
Let $m\in\left\{1,\dots,k\right\}$ and $0<\alpha<1$. 
Suppose Condition \ref{crit:ratios_criterion} holds, and that $\frac{\alpha^{2}}{8}$ and $m$ satisfy Condition \ref{crit:alpha_bound_for_redLength_bound}, 
and that $\alpha$ and $m$ satisfy Condition \ref{crit:alpha_bound_for_length_bound}. 
If $m<k$, also assume that $\alpha$ and $m$ satisfy Condition \ref{crit:bound_on_probs}.
Then:
\begin{equation}
\label{eq:bound_on_badb_volume_part1}
	Pr_{\omega}\left(\left|
		\badb_{m,0}\left(\omega\right)
		\right|>
	4\alpha\len\left(\omega\right)\text{ and } \omega\in\regwalksall\right)
	< 4\alpha
\end{equation}

\end{lem}
\begin{rem}
When this inequality is true for some value of $\alpha$, it is also true for any larger value of $\alpha$. 
On the other hand, Conditions \ref{crit:alpha_bound_for_redLength_bound}, \ref{crit:alpha_bound_for_length_bound} and \ref{crit:bound_on_probs}  essentially serve as lower bounds on $\alpha$, 
so we may think of inequality \eqref{eq:bound_on_badb_volume_part1} as pertaining to the minimal $\alpha$ that satisfies these lower bounds.
\end{rem}
\begin{rem}
Note that $4\alpha$ plays two superficially unrelated roles in \eqref{eq:bound_on_badb_volume_part1}.
This is because we want to use, for any $m$, the same ratio (up to some constant scaling) of $\len\left(\omega\right)$ in inequality \eqref{eq:bound_on_badb_volume_part1}, and get the same probability bounds for all $m$ (again, up to some constant scaling). 
For $m=k$ in particular, this means that $\Pr_{\omega}\left(0\in\badb_{m,0}\left(\omega\right)\right)$ is bounded by $4\alpha$ for $m=k$ (up to some scaling). 
Thus, we want it to be so also for $m<k$. 
But this probability is close to the ratio of $\len\left(\omega\right)$ in inequality \eqref{eq:bound_on_badb_volume_part1}. 
Therefore, the two $4\alpha$-s are related on each other, and they can be taken to be the same without much loss in the quality of our probability bounds.
\end{rem}
\begin{proof}[Proof of Lemma \ref{lem:bad_redlength_local_m_bounds}]
First we handle the case that $m=k$. Note that 
\begin{align}
\label{eq:level_k_badb_prob}
\Pr_{\omega}\left(\left|
	\badb_{k,0}\left(\omega\right)
	\right|>
\alpha\len\left(\omega\right)\text{ and } \omega\in\regwalksall\right) 
& = \Pr_{\omega}\left(0 \in \badb_{k}\left(\omega\right)\right) 
\end{align}
and by using a union bound for $\badb_{k,0}\left(\omega\right)$ as the union of its subsets 
$\badb_{k,0}^{\mathrm{local}}\left(\omega\right)$, 
$\badb_{k,0}^{\mathrm{upper}}\left(\omega\right)$, 
$\badb_{k,0}^{\mathrm{lower}}\left(\omega\right)$, 
$\badb_{k,0}^{\mathrm{redUpper}}\left(\omega\right)$ and  
$\badb_{k,0}^{\mathrm{redLower}}\left(\omega\right)$, 
and then using the bounds from 
Corollary \ref{cor:bound_prob_badb_lower},
Lemma \ref{lem:probablistic lower bound on next}
and Lemma \ref{cor:bound_prob_badb_local},
we get by the assumption of Condition \ref{crit:ratios_criterion} that:
\begin{align*}
\eqref{eq:level_k_badb_prob} <
&	4 L_{m} \exp\left(-\frac{ \beta_{m} R^{\mathrm{lower}}_{m} }{L_{m}}\right) +\\
&+	M_{k}^{\mathrm{lower}}\exp\left(-\frac{\left(L_{1} \cdots L_{k}\right)^{2}}{M_{k}^{\mathrm{lower}}}\right) +
	R_{k}^{\mathrm{lower}}\exp\left(-\frac{L_{k}^{2}}{R_{k}^{\mathrm{lower}}}\right) +\\
&+	\exp\left(-\frac{1}{27}\frac{M_{k}^{\mathrm{lower}}}{\left(L_{1} \cdots L_{k}\right)^{2}}\right) +
	\exp\left(-\frac{1}{27}\frac{R_{k}^{\mathrm{lower}}}{L_{k}^{2}}\right).
\end{align*}
Thus we get from  Conditions \ref{crit:alpha_bound_for_redLength_bound} and \ref{crit:alpha_bound_for_length_bound}, 
and from the fact that $k\ge2$, that:
\begin{align}
\label{eq:level_k_badb_prob2}
\eqref{eq:level_k_badb_prob} <
	\frac{1}{2}\alpha^{2} + 	\frac{1}{2}\alpha + 	\frac{1}{16}\alpha^{2} + 	\left(\frac{1}{4}\alpha\right)^{16} + 	\left(\frac{1}{16}\alpha^{2}\right)^{16} 	< 4\alpha.
\end{align}

Now we may turn to the case that $m<k$.
Since we assume that 
$\frac{\alpha^{2}}{8}$ and $m$ satisfy Condition \ref{crit:alpha_bound_for_redLength_bound}, 
then by Lemma \ref{lem:bad_redlength_local_m_m-1_bounds}:
\begin{multline}
\label{eq:bound_prob_error_density1}
Pr_{\omega}\left(\left|\badb^{\mathrm{local}}_{m,m-1}\left(\omega\right) \cup \badb_{m,m-1}^{\mathrm{redLength}}\left(\omega\right)\right|>
4\alpha^{2}\len_{m-1}\left(\omega\right)\text{ and } \omega\in\regwalksall\right)
\\ < 
2\exp\left(-\frac{1}{64R_{m}^{\mathrm{upper}}}N_{m-1} {\alpha}^{4} \right) +
4\exp\left(-\frac{1}{256}N_{m}\left( \alpha^{2} \frac{R^{\mathrm{lower}}_{m}}{R^{\mathrm{upper}}_{m}} \right)^{2}\right).
\end{multline}
If $m-1=0$ then by Condition \ref{crit:bound_on_probs} we are done. Assume $m>1$. Note that:
\[
\left|\badb^{\mathrm{local}}_{m,0}\left(\omega\right) \cup \badb_{m,0}^{\mathrm{redLength}}\left(\omega\right)\right| = 
\sum_{j\in\badb^{\mathrm{local}}_{m,m-1}\left(\omega\right) \cup \badb_{m,m-1}^{\mathrm{redLength}}\left(\omega\right)} \left|I_{m-1,j}\left(\omega\right)\right|
\]
\[
\len\left(\omega\right) = 
\sum_{j=0}^{\len_{m-1}\left(\omega\right)-1} \left|I_{m-1,j}\left(\omega\right)\right|
\]
and also note that $\left\{\left|I_{m-1,j}\left(\omega\right)\right|\right\}_{j=0}^{\len_{m-1}\left(\omega\right)-1}$ are i.i.d random variables with 
the same probability law as $\next\left(\omega,L_{1}\cdots L_{m-1},0\right)$, where $L_{1}\cdots L_{m-1} \ge L_{1}\ge 1000$.
In Lemma \ref{lem:sum_of_delta_gt_epsilon}, the order of summation of the $X_{i}$-s doesn't matter, so we may apply it with $N=\len_{m-1}\left(\omega\right)$, $\delta=4\alpha^2$ and 
$\left\{\left|I_{m-1,j}\left(\omega\right)\right|: 0\le j \le \len_{m-1}\left(\omega\right)-1\right\} = \left\{X_{i}: 0\le i \le N-1\right\}$, 
 and get
\begin{multline}
\label{eq:bound_prob_error_density2}
Pr_{\omega}\left(\left|\badb^{\mathrm{local}}_{m,0}\left(\omega\right) \cup \badb_{m,0}^{\mathrm{redLength}}\left(\omega\right)\right|>
2\alpha\len\left(\omega\right)\text{ and } \omega\in\regwalksall\right)
\\ < 
Pr_{\omega}\left(\left|\badb^{\mathrm{local}}_{m,m-1}\left(\omega\right) \cup \badb_{m,m-1}^{\mathrm{redLength}}\left(\omega\right)\right|>
4\alpha^{2}\len_{m-1}\left(\omega\right)\text{ and } \omega\in\regwalksall\right) +\\
+\frac{5}{N_{m-1}},
\end{multline}
which by using \eqref{eq:bound_prob_error_density1} gives:
\begin{multline}
\label{eq:bound_prob_error_density3}
\eqref{eq:bound_prob_error_density2} < 
2\exp\left(-\frac{1}{64R_{m}^{\mathrm{upper}}}N_{m-1} {\alpha}^{4} \right) +\\+
4\exp\left(-\frac{1}{256}N_{m}\left( \alpha^{2} \frac{R^{\mathrm{lower}}_{m}}{R^{\mathrm{upper}}_{m}} \right)^{2}\right)+\frac{5}{N_{m-1}}.
\end{multline}

Our assumptions also imply that $M_{m}^{\mathrm{upper}}>1280\left(L_{1}\cdots L_{m}\right)^{2}$, $L_{1}\cdots L_{m}\ge 500$, 
$\alpha>4\exp\left(-\frac{1}{432}\frac{M_{m}^{\mathrm{upper}}}{\left(L_{1}\cdots L_{m}\right)^{2}}\right)$ 
and $\alpha>2 M_{m}^{\mathrm{lower}}\exp\left(-\frac{\left(L_{1}\cdots L_{m}\right)^{2}}{M_{m}^{\mathrm{lower}}}\right)$.
Therefore, by Lemma \ref{lem:bad_lengths_m_bounds}:
\begin{multline*}
\Pr_{\omega}\left(\left|\badb_{m,0}^{\mathrm{length}}\left(\omega\right)\right|>
2\alpha\len\left(\omega\right)
\text{ and } \omega\in\regwalksall\right) \\<
\exp\left(-\frac{1}{4M_{m}^{\mathrm{upper}}}N_{0} {\alpha}^{2}\right) + \exp\left(-\frac{1}{4}N_{m}\alpha^2\right)\\<
\exp\left(-\frac{1}{4M_{m}^{\mathrm{upper}}}N_{0} {\alpha}^{2}\right) + 
\exp\left(-\frac{1}{256}N_{m}\left( \alpha^2 \frac{R^{\mathrm{lower}}_{m}}{R^{\mathrm{upper}}_{m}} \right)^{2}\right)
\end{multline*}
which together with \eqref{eq:bound_prob_error_density3} implies that 
\begin{multline*}
Pr_{\omega}\Big(\left|\badb^{\mathrm{local}}_{m,0}\left(\omega\right) \cup \badb_{m,0}^{\mathrm{redLength}}\left(\omega\right)
 \cup \badb_{m,0}^{\mathrm{length}}\left(\omega\right)\right|>
4\alpha\len\left(\omega\right)\\\hspace{255pt}\text{ and } \omega\in\regwalksall\Big) \\ 
\begin{aligned}\le&
Pr_{\omega}\left(\left|\badb^{\mathrm{local}}_{m,0}\left(\omega\right) \cup \badb_{m,0}^{\mathrm{redLength}}\left(\omega\right)\right|>
2\alpha\len\left(\omega\right)\text{ and } \omega\in\regwalksall\right) +\\&+
\Pr_{\omega}\left(\left|\badb_{m,0}^{\mathrm{length}}\left(\omega\right)\right|>
2\alpha\len\left(\omega\right)\text{ and } \omega\in\regwalksall\right)
\\ < &
2\exp\left(-\frac{1}{64R_{m}^{\mathrm{upper}}}N_{m-1} {\alpha}^{4} \right) +
5\exp\left(-\frac{1}{256}N_{m}\left( \alpha^{2} \frac{R^{\mathrm{lower}}_{m}}{R^{\mathrm{upper}}_{m}} \right)^{2}\right) + \\
&+\exp\left(-\frac{1}{4M_{m}^{\mathrm{upper}}}N_{0} {\alpha}^{2}\right)
+\frac{5}{N_{m-1}}.
\end{aligned}
\end{multline*}
By the assumption of Condition \ref{crit:bound_on_probs}, the sum on the right hand side is bounded by $4\alpha$, which concludes the proof of this lemma.
\end{proof}

\subsection{Bounding the probability that $\omega\in\regwalksall$}

\begin{lem}
\label{lem:lower bound on regwalksall prob}

It holds that:
\[
\Pr_{\omega}\left(\omega\notin\regwalksall\right) 
  <   \sum_{m=1}^{k} N_{m-1}\exp\left(-\frac{\left(L_{m}\cdots L_{k}\right)^{2}}{N_{m-1}}\right).
\]
\end{lem}
\begin{proof}
By Lemma \ref{lem:probablistic lower bound on next}, for any $m\in\left\{0,\dots,k-1\right\}$ it holds that 
													   
\[
\Pr_{\omega}\left(\omega\notin\regwalks_{\ge}^{m}\right) < 
N_{m}\exp\left(-\frac{\left(L_{m+1}\cdots L_{k}\right)^{2}}{N_{m}}\right),
\]
and therefore:
\begin{align*}
\Pr_{\omega}\left(\omega\notin\regwalksall\right) 
& =   \Pr_{\omega}\left(\omega\notin\bigcap_{m=0}^{k-1}\regwalks_{\ge}^{m}\right) \\
& \le \sum_{m=0}^{k-1}\Pr_{\omega}\left(\omega\notin\regwalks_{\ge}^{m}\right) \\
& <   \sum_{m=0}^{k-1} N_{m}\exp\left(-\frac{\left(L_{m+1}\cdots L_{k}\right)^{2}}{N_{m}}\right).
\end{align*}
\end{proof}

\begin{cor}
\label{cor:union_of_bads_upper_bound}
Suppose Condition \ref{crit:ratios_criterion} holds, and that for any $m\in\left\{1,\dots,k\right\}$, 
$\frac{\alpha^{2}}{8m^{4}}$ and $m$ satisfy Condition \ref{crit:alpha_bound_for_redLength_bound}, 
and $\frac{\alpha}{m^2}$ and $m$ satisfy Condition \ref{crit:alpha_bound_for_length_bound}, and that:
\begin{equation}
\label{eq:condition_on_N_m}
m^{-2}\alpha>N_{m-1}\exp\left(-\frac{\left(L_{m}\cdots L_{k}\right)^{2}}{N_{m-1}}\right).
\end{equation}
Also assume that for any $m\in\left\{1,\dots,k-1\right\}$, $\frac{\alpha}{m^2}$ and $m$ satisfy Condition \ref{crit:bound_on_probs}.

Then it holds that:
\[
Pr_{\omega}\left(\left|\cup_{m=1}^{k}
	\badb_{m,0}\left(\omega\right)
	\right|\ge
10\alpha\len\left(\omega\right) \right) 
< 10\alpha.
\]
\end{cor}
\begin{proof}

First note that due to the fact that $\sum_{m=1}^{k}m^{-2}<\sum_{m=1}^{\infty}m^{-2}=\frac{\pi^2}{6}<2$:

\begin{multline}
\label{eq:bound on union of bads}
Pr_{\omega}\left(\left|\cup_{m=1}^{k}
	\badb_{m,0}\left(\omega\right)
	\right|>
8\alpha\len\left(\omega\right)\text{ and } \omega\in\regwalksall\right) \\
\begin{aligned}
<&
Pr_{\omega}\left(\sum_{m=1}^{k}\left|
	\badb_{m,0}\left(\omega\right)
	\right|>
4{\left(\sum_{m=1}^{k}m^{-2}\right)}\alpha\len\left(\omega\right)\text{ and } \omega\in\regwalksall\right)
\\<&\sum_{m=1}^{k}
Pr_{\omega}\left(\left|
	\badb_{m,0}\left(\omega\right)
	\right|>
4{m^{-2}}\alpha\len\left(\omega\right)\text{ and } \omega\in\regwalksall\right) .
\end{aligned}
\end{multline}

Using Lemma \ref{lem:bad_redlength_local_m_bounds} we get:
\begin{equation}
\label{eq:bound_on_union_badb_and_regwalksall_prob}
\eqref{eq:bound on union of bads} < \sum_{m=1}^{k} \frac{4\alpha}{m^2} <8\alpha,
\end{equation}
and then, by Lemma \ref{lem:lower bound on regwalksall prob} and  \eqref{eq:condition_on_N_m}, 
\begin{equation}
\label{eq:bound_on_union_badb_and_regwalksall_prob2}
\Pr_{\omega}\left(\omega\notin\regwalksall\right) 
 <   \sum_{m=1}^{k} N_{m-1}\exp\left(-\frac{\left(L_{m}\cdots L_{k}\right)^{2}}{N_{m-1}}\right)
 <   \sum_{m=1}^{k}  m^{-2}\alpha
 < 	 2\alpha.
\end{equation}
By combining \eqref{eq:bound_on_union_badb_and_regwalksall_prob} with \eqref{eq:bound_on_union_badb_and_regwalksall_prob2} we get 
\begin{equation*}
Pr_{\omega}\left(\left|\cup_{m=1}^{k}
	\badb_{m,0}\left(\omega\right)
	\right|>
8\alpha\len\left(\omega\right) \right) 
< 10\alpha,
\end{equation*}
yielding the required conclusion.

\end{proof}

\subsection{Explicit choice of the parameters}

Using Corollary \ref{cor:union_of_bads_upper_bound}, we will soon set explicit values to the parameters that we left undefined hitherto. Before that, a final easy technical lemma is required:
\begin{lem}
\label{lem:condition_on_alpha}
If $a,b,m,x>0$ and $x>a^{\frac{1}{2}}b^{-\frac{1}{2m}}$, then $x>a\exp\left(-b x^{m}\right)$.
\end{lem}
\begin{proof} This can be seen through calculation. \end{proof}

At long last, we define the parameters:

\begin{defn}
\label{defn:sinwalks}
Let $A,B,k\in\mathbb{N}$ and $0<\alpha<1$ be some fixed parameters. For the following set of parameters for any $m\in\left\{1,\dots,k\right\}$:
\begin{enumerate}
\item	$L_{m} = A m^{B}$
\item	$\frac{M^{\mathrm{upper}}_{m}}{\left(L_{1} \cdots L_{m}\right)^{2}} = 2000\ln\left(\frac{m^2}{\alpha}\right)$
\item	$\frac{\left(L_{1} \cdots L_{m}\right)^{2}}{M^{\mathrm{lower}}_{m}} = 2\ln\left(\left(L_{1} \cdots L_{m}\right)^{2}\right) + 4\ln\left(\frac{m^2}{\alpha}\right)$
\item	$\frac{R^{\mathrm{upper}}_{m}}{L_{m}^{2}} = 4000\ln\left(\frac{m^2}{\alpha}\right)$
\item	$\frac{L_{m}^{2}}{R^{\mathrm{lower}}_{m}} = 2\ln\left(L_{m}^{2}\right) + 16\ln\left(\frac{m^2}{\alpha}\right)$
\item	$\beta_{m}
= 10\cdot\frac{2L_{m}}{R^{\mathrm{lower}}_{m}}\ln\left(
	\frac{200m^{4} R^{\mathrm{upper}}_{m} L_{m}}{ \alpha^{2} R^{\mathrm{lower}}_{m} }
\right)$
\item	$\frac{\left(L_{m+1} \cdots L_{k}\right)^{2}}{N_{m}} = 10\ln\left(L_{m+1} \cdots L_{k}\right) + \ln\left(\frac{m^2}{\alpha}\right)$ where $1\le m \le k$, \\and 
$\frac{\left(L_{1} \cdots L_{k}\right)^{2}}{N_{0}} = 10\ln\left(L_{1} \cdots L_{k}\right)$
\item 	$\alpha = \left(10A\right)^{-4}$.\end{enumerate}
we define the set of \textbf{sinuous walks}:
\begin{equation}
\label{eq:defn_sinwalks}
\sinwalks\left(A,B,k\right) = \left\{ 
	\omega\in\walks:
	\left|\bigcup_{m=1}^{k}\badb_{m,0}\left(\omega\right)\right|<10\alpha\len\left(\omega\right)
		\right\}.
\end{equation}
\end{defn}

The above values $R^{\mathrm{lower}}_{m},R^{\mathrm{upper}}_{m},M^{\mathrm{lower}}_{m},M^{\mathrm{upper}}_{m}$ are rounded to the closest integer (if $A\ge1000$, 
the rounded value will be between half and twice the original, and that is sufficient for us).

For the upper bounds on the probabilities concerning $M^{\mathrm{lower}}_{m},R^{\mathrm{lower}}_{m}$ and $N_{m-1}$ to be meaningful, 
we should require that the value that we give these parameters is at least $1$. that is, 
$\frac{\left(L_{1} \cdots L_{m}\right)^{2}}{ 2\ln\left(\left(L_{1} \cdots L_{m}\right)^{2}\right) + 4\ln\left(\frac{m^2}{\alpha}\right) } > 1$, 
$\frac{L_{m}^{2}}{2\ln\left(L_{m}^{2}\right) + 16\ln\left(\frac{m^2}{\alpha}\right)} > 1$, and
$\frac{\left(L_{m} \cdots L_{k}\right)^{2}}{5\ln\left(\left(L_{m} \cdots L_{k}\right)^{2}\right) + \ln\left(\frac{\left(m-1\right)^2}{\alpha}\right)} > 1$
. 
Since for $x\ge1000$ (such as $A,L_{1},\dots L_{k}$) it holds that $x-2\ln\left(x\right)>\frac{x}{2}$ and $x-5\ln\left(x\right)>\frac{x}{2}$, it is sufficient to have:
\[
\ln\left(\frac{\left(m-1\right)^2}{\alpha}\right) < \ln\left(\frac{m^2}{\alpha}\right) < \frac{L_{m}^{2}}{16} 
< \min\left(\frac{\left(L_{1} \cdots L_{m}\right)^{2}}{8} , \frac{\left(L_{m} \cdots L_{k}\right)^{2}}{2}\right). 
\]

By this reduction, it is enough to have $\frac{m^2}{\alpha} < \exp\left(\frac{L_{m}^{2}}{16}\right)$. 
Since this is required for each $m$, and the sequence $L_{m}$ increases at least as fast as $m$, it is sufficient 
to require that it holds just for $m=1$, where $L_{m}=A$. So our condition on $\alpha$ is that $\alpha > \exp\left(-\frac{A^{2}}{16}\right)$, 
which holds for $\alpha = \left(10A\right)^{-4}$. 
\begin{prop}
\label{prop:prob_of_sinuosity}
Let $A,B,k\in\mathbb{N}$ be such that $A\ge1000,B\ge250,k\ge20A$. Then it holds that:
\[
\Pr\left(\omega\notin\sinwalks\left(A,B,k\right)\right) < 
\frac{1}{1000A^{4}}
\]
and $\alpha<\frac{1}{10000M^{\mathrm{upper}}_{1}}$.\end{prop}
\begin{proof}
We explained why the choice of parameters in $\sinwalks\left(A,B,k\right)$ is meaningful just before this Proposition.
It remains to simply verify that each of the conditions of Corollary \ref{cor:union_of_bads_upper_bound} hold for the choice of parameters in Definition \ref{defn:sinwalks}.
\begin{enumerate}
\item 
First note that indeed $\alpha=\left(10A\right)^{-4}<\frac{1}{20M^{\mathrm{upper}}_{1}}$ by (8) in Definition \ref{defn:sinwalks}.
Since $A\ge1000$:
\[
M^{\mathrm{upper}}_{1}\alpha <
4000L_{1}^{2}\ln\left(\frac{1}{\alpha}\right)\left(10A\right)^{-4} =
0.4 \cdot 4\ln\left(10A\right) A^{-2} <
10^{-4}.
\]
\item 
Condition \ref{crit:ratios_criterion} holds for any $1\le m \le k$ and $0<\frac{\alpha}{m^2}<\frac{1}{10^{4} M^{\mathrm{upper}}_{1} m^2}$, since $A\ge1000$ by 
items (2) and (4) in Definition \ref{defn:sinwalks}.
\item 
Condition \ref{crit:alpha_bound_for_length_bound} holds for any $1\le m \le k$ and $0<\frac{\alpha}{m^2}<\frac{1}{10^{4} M^{\mathrm{upper}}_{1}  m^2}$, 
by items (2) and (4) in Definition \ref{defn:sinwalks}.
\item
Condition \ref{crit:alpha_bound_for_redLength_bound} holds for any $1\le m \le k$ and $0<\frac{\alpha^2}{8m^4}<\frac{1}{8\cdot10^{8} \left(M^{\mathrm{upper}}_{1}\right)^2  m^4}$, 
by items (4), (5) and (6) in Definition \ref{defn:sinwalks}.
\item 
Inequality \eqref{eq:condition_on_N_m} holds for any $1\le m \le k$ and $0<\frac{\alpha}{m^2}<\frac{1}{10^{4} M^{\mathrm{upper}}_{1}  m^2}$, 
by items (7) in Definition \ref{defn:sinwalks}.
\item
To show that Condition \ref{crit:bound_on_probs} holds for any $1\le m \le k-1$ and $0<\frac{\alpha}{m^2}<\frac{1}{10^{4} M^{\mathrm{upper}}_{1}  m^2}$, 
note that by Lemma \ref{lem:condition_on_alpha} 
it is sufficient to see that for any $1\le m\le k-1$ the number $\frac{\alpha}{m^2}$ is larger than
$4 \left(\frac{N_{m-1}}{R^{\mathrm{upper}}_{m}}\right)^{-\frac{1}{4}}$, 
$2\sqrt{5} \left(N_{m}\left(\frac{R^{\mathrm{lower}}_{m}}{R^{\mathrm{upper}}_{m}}\right)^{2}\right)^{-\frac{1}{8}}$, 
$\frac{1}{\sqrt{2}} \left(\frac{N_{0}}{M^{\mathrm{upper}}_{m}}\right)^{-\frac{1}{4}}$ and 
$5\left(N_{m-1}\right)^{-1}$. The case that $m=k-1$ implies all of the cases $m<k-1$, and in this case the four expressions are bounded from above 
by $\frac{1000 B k}{ {L_{k}}^{\frac{1}{8}}}$,
which is bounded by $\frac{\alpha}{k^2}$ since Definition \ref{defn:sinwalks} implies:
\begin{align*}
k^{2} \cdot 1000 B k L_{k}^{-\frac{1}{8}} & = 
1000 B k^{3} \left(A k^{B}\right)^{-\frac{1}{8}} =
1000 A^{-\frac{1}{8}} B k^{3-\frac{B}{8}} \\ & <
1000 A^{-\frac{1}{8}} \cdot 80k^{\frac{B}{80}} \cdot k^{3-\frac{B}{8}} = 
80000 A^{-\frac{1}{8}} \left(20A\right)^{3-\frac{9B}{80}} \\ & = 
80000 \cdot 2^{\frac{240-9B}{80}}  \left(10A\right)^{\frac{230-9B}{80}} <
2^{\frac{1840-9B}{80}}  \left(10A\right)^{\frac{230-9B}{80}} \\ & <
\left(10A\right)^{-4} = 
\alpha.
\end{align*}
\end{enumerate}

\end{proof}

\section{The Set of Tests \texorpdfstring{$\Lambda_{k}$}{Lambda k}}
\label{section_Lambda}

\subsection{Definition of $\Lambda_{k}$ and Outline of the Proof of Theorem \ref{thm:main_theorem}}
\label{subsection_def_of_tests}

In Definition \ref{def:test} we defined what a test is, 
and in Theorem \ref{thm:main_theorem} we claimed the existence of a set of tests denoted by $\Lambda$, that has some desirable properties. We now construct this set of tests recursively, based on a choice of a yet another sequence of parameters $B_m$ (for $m\geq 2$) which will be called the \textbf{branching number at level $m$}.

\begin{defn}
\label{def:Lambda}
We define $\Lambda_{m}$, the \textbf{set of level-$m$ tests},  for any $m\in\mathbb{N}$ by recursion.
A \textbf{level-$1$ test} is a sequence of length $1$ of the form $\left(\left(t,\Delta\right)\right)=\left(\left(0,0\right)\right)$.
The set of level-$1$ tests is $\Lambda_{1}$.
For any $m\in\mathbb{N}$, we define recursively:
 \[
 \Lambda_{m+1}\coloneqq\left\{ \left((t+t_{j},\Delta+2\Delta_{j}):
	\begin{array}{c} 
		1\leq j\leq B_{m+1}\\[3pt]
		(t,\Delta)\in T_{j}
	\end{array}
	\right)
	\middle|
	\begin{array}{c}
	    \forall 1\leq j \leq B_{m+1},\\[3pt]
		1 \leq t_j \leq M_{m+1}^{\mathrm{upper}} \text{ distinct}, \\[3pt]
		|\Delta_{j}| \leq M_{m+1}^{\mathrm{upper}},\\[3pt]
		T_j \in \Lambda_{m}
	\end{array} 
\right\} 
 \]
and a \textbf{level-$(m+1)$ test} is 
simply an element of $\Lambda_{m+1}$.

\end{defn}
\medskip

\noindent
The set of level-$m$ tests $\Lambda _ m$ depends implicitly on the parameters $A, B, k$ of Definition~\ref{defn:sinwalks}. When we want to make this dependence explicit, we write it as $\Lambda _ m (A, B, k)$.

\medskip

For Theorem \ref{thm:main_theorem}, we will use the set of tests $\Lambda=\Lambda_{k}\left(A,B,k\right)$. 
We hope that at least one of the tests in $\Lambda_k$, a test we called $\lambda_{\rm passed}$ in Theorem \ref{thm:main_theorem}, will satisfy 
the consequences of this theorem. 
In particular, we want to have $|\lambda_{\rm passed}(\omega')|>N^\theta$. 
We achieve this by requiring that for any $\left(t_{1},\Delta_{1}\right),\left(t_{2},\Delta_{2}\right)\in\lambda_{\mathrm{passed}}$, if they are different, 
then the following intervals are disjoint:
\begin{equation}
\label{eq:disjoint intervals}
\left\{X^{\omega}_{t} : t\in I'_{t_{1}}\right\} \cap
\left\{X^{\omega}_{t} : t\in I'_{t_{2}}\right\} = \emptyset
\end{equation}
where $I'_{t} = \left\{i_{0}+t,\dots,\next\left(\omega,L_{1},i_{0}+t\right)\right\}$,
and additionally that the number of such intervals 
 will be large enough.

We will use the following strategy to ensure the disjointness in \eqref{eq:disjoint intervals}:
Pick $B_k$ intervals $I_{k,j}^{0}\left(\omega\right)$ such that the scenery intervals $\left\{X^{\omega}_{t}:t\in I_{k,j}^{0}\left(\omega\right)\right\}$ 
are pairwise disjoint. Inside each of these intervals, pick $B_{k-1}$ intervals $I_{k-1,j}^{0}\left(\omega\right)$ such that the scenery intervals 
$\left\{X^{\omega}_{t}:t\in I_{k-1,j}^{0}\left(\omega\right)\right\}$ are pairwise disjoint, and note that we now have $B_{k} B_{k-1}$ intervals that have 
pairwise disjoint scenery intervals. 
Repeat this process until you have $B_{k} \cdots B_{2}$ intervals $I_{1,j}^{0}\left(\omega\right)$ with pairwise disjoint scenery intervals, which is what we needed 
for \eqref{eq:disjoint intervals}.
A similar strategy was used in \cite{Kalikow}. Also relevant are the discrete Cantor sets from \cite{Austin}.

To have any chance of achieving the pairwise disjointness of the $B_{m}$ intervals from the $m$-th stage of the strategy above, 
the branching number at level $m$ needs to be less than $L_m$.
This is because for each interval $I_{m,j}^{0}\left(\omega\right)$, the scenery interval must contain an interval of the form 
\[\left\{L_{1}\cdots L_{m} \cdot x,
\dots ,L_{1}\cdots L_{m} \cdot \left(x+1\right)\right\},\] 
for some $x$, so this scenery interval can pack at most $\frac{1}{2}L_{m}$ pairwise disjoint scenery intervals of the form
\[\left\{L_{1}\cdots L_{m-1} \cdot \left(x'-1\right), \dots ,L_{1}\cdots L_{m-1} \cdot \left(x'+1\right)\right\},\] 
and the scenery intervals $I_{m-1,j'}^{0}\left(\omega\right)$ could be this large, in the extreme case. 

Now, given that idea, we would like to maximize the $B_{m}$-s, while ensuring that no matter which places in the record $\left(\left(\omega_{i},\sigma\left(X^{\omega}_{i}\right)\right)\right)_{i=0}^{l-1}$ the adversary chooses to corrupt, we will still have 
some $B_{m}$ relatively uncorrupted intervals. Suppose we are the adversary, and suppose we only change the viewed scenery. A good strategy for us would be to 
change the viewed scenery at the least visited places in the scenery interval that the record covers. This signifies that the possibility of reconstruction is 
controlled by the behaviour of the local time measure of $\omega$ (as discussed in subsection \ref{subsection:local_time}). 
In Lemma \ref{inner lemma : J ratio} below, we will translate our assumption on the local time measure in Definition \ref{def:badblocal} to what is essentially a lower bound on the maximal $B_{m}$ that we can use, and use the choice of $B_{m}$ that we get from it:
\begin{equation}
\label{eq:premise B beta small}
B_{m} \coloneqq \left\lfloor \frac{M^{\mathrm{lower}}_{m}}{\beta_{m} M^{\mathrm{upper}}_{m-1} R^{\mathrm{upper}}_{m} \left(m^2+1\right)} \right\rfloor
.\end{equation}

Later, we will choose the parameters $A,B$ in Definition \ref{defn:sinwalks} (which we have not done yet) so that $B_{m}$ will be roughly $L_{m}^{2\theta}$, for $2\theta$ slightly less than $1$, as in Theorem~\ref{thm:main_theorem}.

\medskip

\subsection{Properties of $\Lambda_{k}$}
The following proposition gives an upper bound on the number of tests in $\Lambda_m$ in terms of the parameters $A$ and $B_2$,\dots,$B_m$. 
This will be used to establish consequence (3) of Theorem \ref{thm:main_theorem}, 
i.e that $\ln\left|\Lambda\right| < \epsilon \left|\lambda_{\mathrm{passed}}\left(\omega'\right)\right|$
.

\begin{prop}
\label{prop:Bound_on_size_of_lambda}
Assume that $B\ge16$, $A\ge2^{200}$ and $B\le32A^{\frac{1}{32}}$. Then for any 
$m\in
\left\{2,3,\dots\right\}$
it holds that  $\ln\left|\Lambda_{m}\left(A,B,k\right)\right|<
10^{6}A^{\frac{1}{32}} B_{2}\cdots B_{m}$.
\end{prop}
\begin{proof}
First note that $\left|\Lambda_{1}\right|=1$, so the claim holds for $m=1$. 
Now, using the recursive definition:
\begin{align*}
\left|\Lambda_{m+1}\right| & \leq \left(M_{m+1}^{\mathrm{upper}}\right)^{B_{m+1}}\left(2M_{m+1}^{\mathrm{upper}}+1\right)^{B_{m+1}}\left|\Lambda_{m}\right|^{B_{m+1}}\\
 & <\left(\left(M_{m+1}^{\mathrm{upper}}\right)^{3}\left|\Lambda_{m}\right|\right)^{B_{m+1}}
 \end{align*}
 and hence
\[
\ln\left|\Lambda_{m+1}\right|<B_{m+1}\left(3\ln\left(M_{m+1}^{\mathrm{upper}}\right)+\ln\left|\Lambda_{m}\right|\right).
\]
Thus we get by induction that for any $m\in\left\{2,3,\dots\right\}$:
\begin{multline}
\label{eq:log_Lambda_bound}
\begin{aligned}
\ln\left|\Lambda_{m}\right| &<
3{B_{m}}\ln\left(M_{m}^{\mathrm{upper}}\right) + 3{B_{m}B_{m-1}}\ln\left(M_{m-1}^{\mathrm{upper}}\right) 
+\cdots \\&\quad \cdots+
3\left(B_{m}\cdots B_{2}\right)\ln\left(M_{2}^{\mathrm{upper}}\right) 
\\&=
B_{2}\cdots B_{m}\sum_{n=2}^{m} \frac{3\ln\left(M_{n}^{\mathrm{upper}}\right)}{B_{2}\cdots B_{n-1}}
\end{aligned}
\end{multline}
(where we take the denominator of the above sum for $n=2$ to be 1).

To finish the proof, it is sufficient to see that the sum is bounded by $100\ln A$. 
Using the choice of parameters in Definition \ref{defn:sinwalks} and the definition of $B_m$ in~\eqref{eq:premise B beta small}, 
we get that 
\begin{equation}
\label{eq:bidirectional_B_m_bound}
\frac{L_{m}}{\left(\left(\ln A +B\right)m\right)^{13}}<B_{m}<L_{m}
\end{equation}
and by using the definitions of $M^{\mathrm{upper}}_{m}$ and $\alpha$, and inserting the value $L_{1}\cdots L_{m} = A^{m} \left(m!\right)^{B}$, we get that 
\begin{multline}
\label{eq:Mupper_upper_bound}
\ln\left(M^{\mathrm{upper}}_{m}\right) = 2\ln\left(L_{1}\cdots L_{m}\right) + \ln\left(2000\ln\left(\frac{m^2}{\alpha}\right)\right)
\\ < 2\ln\left(L_{1}\cdots L_{m}\right) + 13\ln\left(A\right)m
= 15m\ln{A} + 2B\ln{\left(m!\right)}.
\end{multline}
Thus $3\ln\left(M^{upper}_{m}\right)<100\left(\left(\ln A +B\right)m\right)^{2}$,
and therefore for $n\ge1$:
\begin{equation}
\label{eq:bound_on_lower_level_entropy_bound_sum_term}
\frac{6\ln\left(M^{upper}_{n+1}\right)}{B_{2} \cdots B_{n}} < \frac{ 100\left(\left(\ln A + B\right)^{n-1} n!\right)^{16} }{L_{2} \cdots L_{n}} = \frac{ 100\left(\left(\ln A + B\right)^{n-1} n!\right)^{16} }{A^{n-1} \left(n!\right)^{B}}.\end{equation}
Now, since $B\ge16$, $A\ge2^{200}$ and $B\le32A^{\frac{1}{32}}$  by the assumptions of this Proposition, and since $\ln{A}\le 32 A^{\frac{1}{32}}$, then
$\eqref{eq:bound_on_lower_level_entropy_bound_sum_term} < 100\cdot2^{-4\left(n-1\right)}$  for $n\ge2$
. So, by \eqref{eq:Mupper_upper_bound}:
\begin{align*}
\eqref{eq:log_Lambda_bound}
  &<B_{2}  \cdots B_{m}\left(
    3\ln\left(M^{\mathrm{upper}}_{2}    \right)
  +100\sum_{n\ge1} 2^{-4n}\right)
\\&<B_{2}\cdots B_{m}\left(
400\left(\ln{A}+B\right)^{2}
+ \frac{100}{15}\right)
\\&<10^{6}A^{\frac{1}{32}} B_{2}\cdots B_{m}.
\end{align*}
\end{proof}

The following lemma bounds from below the size of a test $\lambda\in\Lambda_m$, as defined in Definition \ref{def:test}. 
Together with the disjointness of the intervals in \eqref{eq:disjoint intervals}, 
this gives an lower bound on the size of reconstructed scenery for the desired test $\lambda_{\mathrm{passed}} \in \Lambda_m$  
in terms of the parameters $B_2$,\dots,$B_m$, and also alternatively using the parameters $B,M^{\mathrm{upper}}_{m}$. 
This will be used to establish consequence (2) of Theorem \ref{thm:main_theorem}, 
i.e that $\left|\lambda_{\mathrm{passed}}\left(\omega'\right)\right| \ge N^{\theta}$ 
($N$, the length of the random walk in this theorem, is roughly $M^{\mathrm{upper}}_{m}$ for $m=k$).
Additionally, this lemma will be used to establish consequence (3) of Theorem \ref{thm:main_theorem}, together with Proposition \ref{prop:Bound_on_size_of_lambda}.

\begin{lem}
\label{lem:how_much_psila}
Assume that $B\ge13$ and $A\ge20000B$. 
For any $m\in\mathbb{N}$ and for any $\lambda\in\Lambda_{m}$,
it holds that $\left|\lambda\right|=B_{2} \cdots B_{m}$.

In particular, if $m\ge20A$ then for any $\lambda\in\Lambda_{m}\left(A,B,k\right)$, 
$\frac{\ln\left|\lambda\right|}{\ln\left(M^{\mathrm{upper}}_{m}\right)} > \frac{1}{2}-\frac{200}{B}$.
\end{lem}
\begin{proof}
Since Definition \ref{def:Lambda} is recursive, we shall prove this lemma by induction.
If $m=1$ then 
for
any $\lambda\in\Lambda_{m}$ , $\left|\lambda\right|=1$
. Suppose that the lemma holds for $m$ and let $\lambda$ be in $\Lambda_{m+1}$. Then 
\[
		\lambda=\left(\left(t+t_{j},\Delta+\Delta_{j}\right):j\in\left\{ 1,\dots,B_{m+1}\right\} ,\left(t,\Delta\right)\in \lambda_{j}\right)
\]
for some $\left\{ t_{1},\dots,t_{B_{m+1}}\right\} $ that are distinct from each other, some
$\Delta_{1},\dots,\Delta_{B_{m+1}} $ , and some $\lambda_{1},\dots,\lambda_{B_{m+1}} \in \Lambda_{m}$.
Thus by induction:
\[
\left|\lambda\right|=\sum_{j=1}^{B_{m+1}}\left|\lambda_{j}\right|=B_{m}\left|\lambda_{1}\right|=B_{m+1} B_{m} \cdots B_{2}, \]
and therefore the first assertion of the lemma holds.

Now, by \eqref{eq:bidirectional_B_m_bound} and the assumptions on $A$ and $B$,
\begin{align*}
\ln\left|\lambda\right| >& \ln\left(L_{2}\cdots L_{m}\right) - \left(13\left(m-1\right)\ln\left(\ln A +B\right) +13\ln\left(\frac{m!}{1!}\right)\right)
\\ >& \ln\left(L_{2}\cdots L_{m}\right) - \left(13m\ln\left(1.00005 A\right) +13\ln\left(m!\right)\right)
\\ >& \ln\left(L_{2}\cdots L_{m}\right) - \left(26\ln\left(A\right)m +13\ln\left(m!\right)\right).
\end{align*}

By combining this with \eqref{eq:Mupper_upper_bound}, inserting the value $L_{2}\cdots L_{m} = A^{m-1} \left(m!\right)^{B}$, and using elementary manipulation of inequalities, we get:
\begin{equation}
\label{eq:psila_Count_Ratio}
\frac{\ln\left|\lambda\right|}{\ln\left(M^{\mathrm{upper}}_{m}\right)} 
> \frac{1}{2} - 100\cdot\frac{\ln\left(A\right)m+\ln\left(m!\right)}{\ln\left(A\right)m+B\ln\left(m!\right)}
> \frac{1}{2} - \frac{100}{B}\left(\frac{\ln\left(A\right)}{\ln\left(m!\right)}+1\right),
\end{equation}
and since $m\ge20A$, then $\eqref{eq:psila_Count_Ratio}>\frac{1}{2}-\frac{200}{B}$.
\end{proof}

\subsection{Existence of a satisfied test in the set $\Lambda_{k}$}
\label{section_some_test_satisfied}
The following proposition will be used to establish consequence (1) of Theorem \ref{thm:main_theorem}, 
i.e that there exists a test $\lambda_{\mathrm{passed}}\in\Lambda_m$ such that 
the adversarially changed record of the observed scenery and walk, $x'$, passes the test $\lambda_{\mathrm{passed}}$ relative to the unchanged scenery $\sigma$. 
In simpler terms, consequence (1) is that the scenery $\sigma$ can be reconstructed from $x'$, while consequences (2) and (3) ensure that this reconstruction is 
large and not achieved by sheer luck. 
The places where $x$ and $x'$ from Theorem \ref{thm:main_theorem} are different is denoted by $E$, and thus in order for the test $\lambda_{\mathrm{passed}}$ to 
be as in Definition \ref{def:passes_test} we want to ensure two things: 
(1) that, in the notations of Definition \ref{def:passes_test}, $X^{\omega'}_{t}+\Delta_{i} = X^{\omega}_{t}$; 
and (2) that this $t$ is not in $E$. These two requirements, up to some relaxations, are the consequences (1) and (2) 
of Proposition \ref{prop:existence_of_satisfied_test}:

\begin{prop}
\label{prop:existence_of_satisfied_test}
Let $m\in\left\{2,\dots,k\right\}$, $\omega\in\walks$, $j\in\left\{0,1,\dots,\len_{m}\left(\omega\right)\right\}$, 
and let $E^{-},E^{+}$ be disjoint subsets of $E\subset\left\{0,1,\dots,\len\left(\omega\right)\right\}$. 
Denote $i_{0} = i_{\omega,L_{1}\cdots L_{m}}\left(j\right)$. 
Assume that for some $0\le\alpha<\left(\prod_{r=2}^{m}\left(1+\frac{1}{r^2}\right)\right)^{-1}$:
\[
\left|I_{m,j}^{0}\left(\omega\right) \cap \left(
	\cup_{i=1}^{k}\badb_{i,0}\left(\omega\right) \cup E
\right)\right| < \alpha \left|I_{m,j}^{0}\left(\omega\right)\right|
\]
Then there exists $\lambda\in\Lambda_{m}$ such that for any $\left(t,\Delta\right)\in\lambda$ it holds that:
\begin{enumerate}
\item 
$
\Delta = 
2\left| E^{+} \cap \left\{i_{0},\dots,i_{0} + t\right\} \right| - 
2\left| E^{-} \cap \left\{i_{0},\dots,i_{0} + t\right\} \right|
$.
\item 
$\left|I'_{t}\cap E \right| < 
	\left(\prod_{r=2}^{m}\left(1+\frac{1}{r^2}\right)\right)\alpha \left|I'_{t}\right|$
, for $I'_{t} = \left\{i_{0}+t,\dots,\next\left(\omega,L_{1},i_{0}+t\right)\right\}$.
\item $M^{\mathrm{lower}}_{1} \le \left|I'_{t}\right| \le M^{\mathrm{upper}}_{1}$.
\end{enumerate}
Also, For any $\left(t_{1},\Delta_{1}\right),\left(t_{2},\Delta_{2}\right)\in\lambda$, if they are different, then:
\[\left\{X^{\omega}_{t} : t\in I'_{t_{1}}\right\} \cap
\left\{X^{\omega}_{t} : t\in I'_{t_{2}}\right\} = \emptyset\] 
and additionally:
\[\left|\left\{
X^{\omega}_{t} : t\in I'_{t_{1}} , \left(t_{1},\Delta_{1}\right)\in\lambda
\right\}\right| \ge L_{1} \cdot \left|\lambda\right|.\]
\end{prop}
We are going to prove this by induction on $m$, increasing from $m=2$. 
We follow the proof strategy described at Subsection \ref{subsection_def_of_tests}.
We begin by finding $B_{m}$ intervals of the form $I_{m,j}^{0}\left(\omega\right)$ with pairwise disjoint scenery intervals, such that each of them satisfies the 
requirements of Proposition \ref{prop:existence_of_satisfied_test} for $m-1$. By applying the Proposition to each of them, we will prove this proposition for $m$. 
For the sake of efficiency, denote $\cup_{i=1}^{k}\badb_{i,0}\left(\omega\right) \cup E$ by $\badb_{\mathrm{all}}$.

The following Lemma, which will be used in several places in the proof of Proposition \ref{prop:existence_of_satisfied_test}, is the motivation of the definition of the various sets $\badb_{m,n}$-s. 
All the work that we did in section \ref{section_sinwalks}, sums up to Lemma \ref{inner lemma : goodness} and Proposition \ref{prop:prob_of_sinuosity}, 
which will be all that we need for the rest of the proof.
\begin{lem}
\label{inner lemma : goodness}
Let $m\in\left\{1,\dots,k\right\}$, $\omega\in\walks$ and $j\in\left\{0,1,\dots,\len_{m}\left(\omega\right)\right\}$.
If:
\begin{equation}
\label{eq:premise_of_goodness_lemms}
\left|I_{m,j}^{0}\left(\omega\right) \cap \badb_{\mathrm{all}}\right| < \left|I_{m,j}^{0}\left(\omega\right)\right|
\end{equation}
then:
\begin{enumerate}
\item	$M^{\mathrm{lower}}_{m} \le \left|I_{m,j}^{0}\left(\omega\right)\right| \le M^{\mathrm{upper}}_{m}$
\item	$R^{\mathrm{lower}}_{m} \le \left|I_{m,j}\left(\omega\right)\right| \le R^{\mathrm{upper}}_{m}$
\item	$\max_{x\in\mathbb{Z}}\left|\left\{ t\in I_{m,j}\left(\omega\right):X_{t}^{\red\left(\omega,L_{1}\cdots L_{m-1}\right)}=x\right\} \right| \le \frac{\beta_{m}}{2}\left|I_{m,j}\left(\omega\right)\right|$
\end{enumerate}
\end{lem}
\begin{proof}
The premise implies that:
\[
\left|I_{m,j}^{0}\left(\omega\right) \cap \left(
	\cap_{i=1}^{k}\goodb_{i,0}\left(\omega\right) \cap
	E^{c}
\right)\right| > 0
\]
and thus:
\[
j \in \goodb_{m,0}\left(\omega\right) = 
\goodb^{\mathrm{length}}_{m,0}\left(\omega\right) \cap \goodb^{\mathrm{redLength}}_{m,0}\left(\omega\right) \cap
\goodb^{\mathrm{local}}_{m,0}\left(\omega\right)
\]
which means, by definitions \ref{def:badbupper}, \ref{def:badbredupper}, \ref{def:badblower}, \ref{def:badbredlower}, \ref{def:badblocal}, that 
Lemma \ref{inner lemma : goodness} holds.
\end{proof}
Note that since $0\le\alpha<\left(\prod_{r=2}^{m}\left(1+\frac{1}{r^2}\right)\right)^{-1}$, a sufficient condition for Lemma \ref{inner lemma : goodness} and \eqref{eq:premise_of_goodness_lemms} is:
\[
\left|I_{m,j}^{0}\left(\omega\right) \cap \badb_{\mathrm{all}}\right| < \left(\prod_{r=2}^{m}\left(1+\frac{1}{r^2}\right)\right)\alpha\left|I_{m,j}^{0}\left(\omega\right)\right|.
\]
Now we may define the following set:
\begin{equation}
\label{eq:def_of_J}
J \coloneqq \left\{ i\in I_{m,j}\left(\omega\right) : 
	\left|I_{m-1,i}^{0}\left(\omega\right) \cap \badb_{\mathrm{all}}\right| < \left(1+\frac{1}{m^2}\right) \alpha\left|I_{m-1,i}^{0}\left(\omega\right)\right|
\right\}.
\end{equation}
\begin{lem}
\label{inner lemma : J ratio}
Let $m\in\left\{2,\dots,k\right\}$, $\omega\in\walks$ and $j\in\left\{0,1,\dots,\len_{m}\left(\omega\right)\right\}$.
If:
\[
\left|I_{m,j}^{0}\left(\omega\right) \cap \badb_{\mathrm{all}}\right| < \alpha\left|I_{m,j}^{0}\left(\omega\right)\right|
\]
then for the $J$ defined above in \eqref{eq:def_of_J}:
\begin{equation}
\label{eq:upper_bound_on_J}
\frac{\left|J\right|}{\left|I_{m,j}\left(\omega\right)\right|} > \frac{M^{\mathrm{lower}}_{m}}{ \left(m^{2} + 1\right) M^{\mathrm{upper}}_{m-1} R^{\mathrm{upper}}_{m} }.
\end{equation}
\end{lem}
\begin{proof}
By Dividing $I_{m,j}^{0}$ 
into sub-intervals $I_{m-1,i}^{0}$ for $i\in I_{m,j}$, we get:
\begin{align*}
\left|I_{m,j}^{0}\left(\omega\right) \cap \badb_{\mathrm{all}}\right| &=
\sum_{i\in J}\left|I_{m-1,i}^{0}\left(\omega\right) \cap \badb_{\mathrm{all}}\right| +
\sum_{i\in I_{m,j}\left(\omega\right) - J}\left|I_{m-1,i}^{0}\left(\omega\right) \cap \badb_{\mathrm{all}}\right| \\ &\ge
\sum_{i\in I_{m,j}\left(\omega\right) - J} \left(1+\frac{1}{m^2}\right) \alpha \left|I_{m-1,i}^{0}\left(\omega\right)\right| \\ &=
-\left(1+\frac{1}{m^2}\right) \alpha  \sum_{i\in J}\left|I_{m-1,i}^{0}\left(\omega\right)\right| +
\left(1+\frac{1}{m^2}\right) \alpha \left|I_{m,j}^{0}\left(\omega\right)\right|.
\end{align*}
As we assumed, $\left|I_{m,j}^{0}\left(\omega\right) \cap \badb_{\mathrm{all}}\right| < \alpha \left|I_{m,j}^{0}\left(\omega\right)\right|$. Thus:
\[
\left(1+\frac{1}{m^2}\right) \alpha  \sum_{i\in J}\left|I_{m-1,i}^{0}\left(\omega\right)\right| >
\frac{1}{m^2} \alpha \left|I_{m,j}^{0}\left(\omega\right)\right|
\]
and by applying Lemma \ref{inner lemma : goodness} to both sides we get:
\[
\left(1+\frac{1}{m^2}\right) \alpha  \left|J\right| M^{\mathrm{upper}}_{m-1} >
\frac{1}{m^2} \alpha M^{\mathrm{lower}}_{m}.
\]
Since, also by Lemma \ref{inner lemma : goodness}, $\left|I_{m,j}\left(\omega\right)\right| \le R^{\mathrm{upper}}_{m}$, we get the conclusion 
of Lemma \ref{inner lemma : J ratio}.

\end{proof}

\begin{proof}[Proof of Proposition \ref{prop:existence_of_satisfied_test}]
By the definition of $B_{m}$ at \eqref{eq:premise B beta small}, Lemma \ref{inner lemma : J ratio} implies that 
\begin{equation}
\label{eq: J I ratio bound compound}
\left|J\right| > \beta_{m} B_{m}\left|I_{m,j}\left(\omega\right)\right|. 
\end{equation}
Denote $X^{J} \coloneqq \left\{ X_{t}^{\red\left(\omega,L_{1}\cdots L_{m-1}\right)} : t \in J\right\}$.
By Lemma \ref{inner lemma : goodness} we have
\begin{align*}
\left| J \right| 
&\le \sum_{x\in X^{J}} \left|\left\{ t\in I_{m,j}\left(\omega\right):X_{t}^{\red\left(\omega,L_{1}\cdots L_{m-1}\right)}=x\right\} \right| \\
&\le \sum_{x\in X^{J}} \frac{\beta_{m}}{2}\left|I_{m,j}\left(\omega\right)\right| \\
&\le \frac{\beta_{m}}{2}\left|I_{m,j}\left(\omega\right)\right| \left|X^{J} \right|
\end{align*}
and thus by \eqref{eq: J I ratio bound compound} it holds that $\left|X^{J} \right|>2 B_{m}$.

Therefore, we may pick $2 B_{m}$ distinct indices in $J$ such that $X_{i}^{\red\left(\omega,L_{1}\cdots L_{m-1}\right)}$ are different from each other.
Among these, there are at least $B_{m}$ different indices $j_{1},\dots,j_{B_{m}} \in J$ such that 
$\left\{X_{t}^{\omega} : t\in I^{0}_{m-1,j_{i}}\left(\omega\right) \right\}$ 
are disjoint from each other.
For $i\in\left\{1,\dots,B_{m}\right\}$, denote:
\[
t_{i} \coloneqq
i_{\omega,L_{1}\cdots L_{m-1}}\left(j_{i}\right) - i_{\omega,L_{1}\cdots L_{m}}\left(j\right)
\]
\begin{multline*}
\Delta_{i} \coloneqq 
\left| E^{+} \cap \left\{i_{\omega,L_{1}\cdots L_{m}}\left(j\right),\dots,i_{\omega,L_{1}\cdots L_{m-1}}\left(j_{i}\right)\right\} \right| \\ -
\left| E^{-} \cap \left\{i_{\omega,L_{1}\cdots L_{m}}\left(j\right),\dots,i_{\omega,L_{1}\cdots L_{m-1}}\left(j_{i}\right)\right\} \right|
\end{multline*}
Since $\left\{i_{\omega,L_{1}\cdots L_{m}}\left(j\right),\dots,i_{\omega,L_{1}\cdots L_{m-1}}\left(t_{i}\right)\right\} \subseteq I_{m,j}^{0}\left(\omega\right)$,
then by Lemma \ref{inner lemma : goodness}
$\left|\Delta_{i}\right| \le \left|I_{m,j}^{0}\right| \le M_{m}^{\mathrm{upper}}$. 
For the same reason, $0 \le t_{i} \le M_{m}^{\mathrm{upper}}$.
Now, as we noted, we will prove the proposition  by induction.

\paragraph{\textbf{Base case}}
Assume that $m=2$, and define:
\[
\lambda = \left(\left(t_{i},2\Delta_{i}\right):i\in\left\{1,\dots,B_{m}\right\}\right).
\]
First note that indeed $\lambda\in\Lambda_{2}$, since as noted above $0 \le t_{i} \le M_{m}^{\mathrm{upper}}$ and $\left|\Delta_{i}\right|\le M_{m}^{\mathrm{upper}}$. 
Denote $i_{0}=i_{\omega,L_{1}L_{2}}\left(j\right)$, and note that $i_{0}+t_{i}=i_{\omega,L_{1}}\left(j_{i}\right)$. This means that, using the notation of the proposition, 
\[
I'_{t_{i}} = \left\{i_{0}+t_{i},\dots,\next\left(\omega,L_{1},i_{0}+t_{i}\right)\right\}=I_{1,j_{i}}^{0}\left(\omega\right)
.
\]
Now, note that the proposition's required properties for $\lambda$ hold. For any $\left(t_{i},2\Delta_{i}\right)\in\lambda$, i.e any $i\in\left\{1,\dots,B_{m}\right\}$:
\begin{enumerate}
    \item
    Property (1) follows directly from the definition of $\Delta_{i}$
    \item 
    Property (2) follows since $j_{i}\in J$ and  $I'_{t_{i}}=I_{1,j_{i}}^{0}\left(\omega\right)$
    \item 
    Property (3) follows from Lemma \ref{inner lemma : goodness} and property (2).
\end{enumerate}
The disjointness of 
$\left\{X_{t}^{\omega} : t\in I_{t_{i}} \right\}=
\left\{X_{t}^{\omega} : t\in I^{0}_{m-1,j_{i}}\left(\omega\right) \right\}$ 
was guaranteed at the choice of $j_{1},\dots,j_{B_{m}}$, and  the lower bound $L_{1}\cdot\left|\lambda\right|$ on the visited scenery follows from the disjointness and the definition of the intervals $I'_{t_{i}}$.

\paragraph{\textbf{Step case}} 
Assume that $m>2$ and that the proposition holds for $m-1$. 
For each $i\in\left\{1,\dots,B_{m}\right\}$, the proposition for $m-1$ implies that there exists $\lambda_{i}\in\Lambda_{m-1}$ that satisfies the properties listed there. To prove the proposition for $m$, we construct $\lambda\in\Lambda_{m}$ that satisfies the required properties as follows:
\[
\lambda = \left(\left(t+t_{i},\Delta+2\Delta_{i}\right) : i\in\left\{1,\dots,B_{m}\right\},\left(t,\Delta\right)\in\lambda_{i}\right).
\]
Let $i\in\left\{1,\dots,B_{m}\right\}$ and $\left(t,\Delta\right)\in\lambda_{i}$, and denote 
$i_{m} = i_{\omega,L_{1}\cdots L_{m}}\left(j\right)$,  
$i_{m-1} = i_{\omega,L_{1}\cdots L_{m-1}}\left(j_{i}\right)$.
\begin{enumerate}
\item 
Conclusion (1) in the proposition for $m-1$ and $j_{i}$ implies that
\[
\Delta = 
2\left| E^{+} \cap \left\{
                i_{m-1}
        ,\dots,
                i_{m-1}+t
    \right\} \right| - 
2\left| E^{-} \cap \left\{
                i_{m-1}
        ,\dots,
                i_{m-1}+t
    \right\} \right|
,\]
and the definition of $\Delta_{i}$ is equivalent to 
\[
\Delta_{i} = 
\left| E^{+} \cap \left\{ i_{m} ,\dots, i_{m-1} \right\} \right| - 
\left| E^{-} \cap \left\{ i_{m} ,\dots, i_{m-1} \right\} \right| ,
\]
thus we get conclusion (1) in the proposition for $m$ and $j$:

\begin{align*}
\Delta+2\Delta_{j} = &
\left| E^{+} \cap \left\{ i_{m} ,\dots, i_{m-1}+t \right\} \right| - 
\left| E^{-} \cap \left\{ i_{m} ,\dots, i_{m-1}+t \right\} \right| \\ = &
\left| E^{+} \cap \left\{ i_{m} ,\dots, i_{m}+t+t_{i} \right\} \right| - 
\left| E^{-} \cap \left\{ i_{m} ,\dots, i_{m}+t+t_{i} \right\} \right|.
\end{align*}

\item 
Denote the intervals that we get from the proposition for $m-1$ and $j_{i}$ by $I'_{t+t_{i}}$. Then:
\begin{align*}
I'_{t+t_{i}}=&
\left\{i_{m-1}+t,\dots,\next\left(\omega,L_{1},i_{m-1}+t\right)\right\}
\\=&
\left\{i_{m}+t+t_{i},\dots,\next\left(\omega,L_{1},i_{m}+t+t_{i}\right)\right\}
\end{align*}
and since at the proposition for $m-1$ and $j_{i}$ we look at $\left|I_{m-1,j_{i}}^{0}\left(\omega\right)\right|$ with a $E$ density of $\left(1+\frac{1}{m^2}\right) \alpha$ (by the definition of $J$), 
then the density in the sub-intervals $I'_{t+t_{i}}$ satisfies by the conclusion (2) in the proposition:
\begin{align*}
\left|I'_{t+t_{i}}\cap E\right| &< 
	\left(\prod_{r=2}^{m-1}\left(1+\frac{1}{r^2}\right)\right) \left(1+\frac{1}{m^2}\right) \alpha \left|I'_{t+t_{i}}\right| \\&= 
	\left(\prod_{r=2}^{m}\left(1+\frac{1}{r^2}\right)\right) \alpha \left|I'_{t+t_{i}}\right|
\end{align*}
thus we got conclusion (2) in the proposition for $m$ and $j$.
\item 
Conclusion (3) in the proposition for $m$ and $j$ and conclusion (3) in the proposition for $m-1$ and $j_{i}$ both say that $M^{\mathrm{lower}}_{1} \le \left|I'_{j_{i},t}\right| \le M^{\mathrm{upper}}_{1}$.
\end{enumerate}

Also, For any $\left(t'+t_{i_{1}},\Delta'+2\Delta_{i_{1}}\right),\left(t''+t_{i_{2}},\Delta''+2\Delta_{i_{2}}\right)\in\lambda$, if they are different, 
then either $i_{1}=i_{2}$, and then the proposition for the case $m-1$ implies
\[\left\{X^{\omega}_{t} : t\in I'_{t'+t_{i_{1}}} \right\} \cap
\left\{X^{\omega}_{t} : t\in I'_{t''+t_{i_{2}}}\right\} = \emptyset,\] 
or $i_{1}\neq i_{2}$, and then 
\begin{align*}
\left\{X^{\omega}_{t} : t\in I'_{j_{i_{1}},t'}  \right\} \subseteq \left\{X_{t}^{\omega} : t\in I^{0}_{m-1,j_{i_{1}}}\left(\omega\right) \right\} \\
\left\{X^{\omega}_{t} : t\in I'_{j_{i_{2}},t''} \right\} \subseteq \left\{X_{t}^{\omega} : t\in I^{0}_{m-1,j_{i_{2}}}\left(\omega\right) \right\}
\end{align*}
and then our choice of $j_{1},\dots,j_{B_{m}}$ guaranteed that the right hand side sets are disjoint. Additionally, the lower bound $L_{1}\cdot\left|\lambda\right|$ on the visited scenery  follows from the disjointness and the definition of the intervals $I'_{t+i_{j}}$, without resort to induction.
\end{proof}

\subsection{Proof of Theorem \ref{thm:main_theorem}}
\label{section_pf_of_main_thm}
Let $A,B\in\mathbb{N}$ be some parameters that will be chosen later in the proof. 
We consider $L_{m}$, et cetera, to be the same parameters from definitions \ref{def:Lambda} and \ref{defn:sinwalks}. 
Let $k$ be the unique integer such that:
\begin{equation}
\label{eq:defn_of_k}
M^{\mathrm{upper}}_{k} < N \le M^{\mathrm{upper}}_{k+1}.
\end{equation}
Note that a level-$k$ test, as in Definition \ref{def:Lambda}, is also a test over length $N$ as in Definition \ref{def:test}.

Suppose $\omega$ is an infinite walk sequence. 
Regardless of whether $\max_{t<N} \left|X^{\omega}_{t}\right|$ is bigger than $L_{1}\cdots L_{k}$ or not, it can be deduced whether it is in 
$\sinwalks\left(A,B,k\right)$ or not by just looking at its first $M^{\mathrm{upper}}_{k}$ steps. 
In light of this, we define $\sinwalks$ to be the set of $\omega\in\left\{\pm1\right\}^{N}$ such that (1) $\max_{t< M^{\mathrm{upper}}_{k}} \left|X^{\omega}_{t}\right|\ge L_{1}\cdots L_{k}$; and (2) if we truncate $\omega$ at $\next\left(\omega,L_{1}\cdots L_{k},0\right)$ we get 
a walk in $\sinwalks\left(A,B,k\right)$. 
By the former argument, the latter definition means that:
\begin{equation}
\label{eq:prob_of_sinwalks_is_unchanged}
\Pr\left(\omega\in\sinwalks\right) = \Pr\left(\omega\in\sinwalks\left(A,B,k\right)\right).
\end{equation}
For the rest of this proof, let $\omega,\omega',x,x'$ be the ones from the statement of the theorem, truncated at $\next\left(\omega,L_{1}\cdots L_{k},0\right)$.

First, note that there exists $N'_{0}$ such that if $N>N'_{0}$ then for the $k$ from \eqref{eq:defn_of_k}, for any $\omega\in\sinwalks\left(A,B,k\right)$ it holds that:
\begin{equation}
\label{eq:flexible_length_rescaling}
\len\left(\omega\right) > N^{\frac{1}{2}+\theta}.
\end{equation}
Since $M^{\mathrm{upper}}_{k} \le N < M^{\mathrm{upper}}_{k+1}$, and since there exists $k_{0}$ such that for $k\ge k_{0}$, 
$M^{\mathrm{lower}}_{k} > \left(M^{\mathrm{upper}}_{k+1}\right)^{\frac{1}{2}+\theta}$
\footnote{Note that $\frac{1}{2}+\theta$ approaches $1$ from below.}
\footnote{It is sufficient to take $k_{0}=20A$, and doesn't impose conditions in addition to what we already have.}
.

Let us define a few more parameters:
\begin{multline}
\label{eq:parameters_for_prop_existence_of_satisfied_test}
\begin{aligned}
m    = & k
\\ E = & \left\{0\le t < \len\left(\omega\right) : x\left(t\right)\neq x'\left(t\right)\right\}
\\ E^{+} = & \left\{t \in E : \left(\omega\left(t\right),\omega'\left(t\right)\right) = \left(+1,-1\right) \right\}
\\ E^{-} = & \left\{t \in E : \left(\omega\left(t\right),\omega'\left(t\right)\right) = \left(-1,+1\right) \right\}
\\ \sinwalks' = & \sinwalks\left(A,B,k\right)
\\ \Lambda = & \Lambda_{k}\left(A,B,k\right)
.
\end{aligned}
\end{multline}

By \eqref{eq:prob_of_sinwalks_is_unchanged} and Proposition \ref{prop:prob_of_sinuosity}, assuming $A\ge1000,B\ge250,k\ge20A$, it holds that:
\begin{equation}
\label{eq:parameter_p_requirement}
\Pr\left(\omega\in\sinwalks\right) > 
1-\frac{1}{1000A^{4}}.
\end{equation}
Let $\omega$ be in $\sinwalks$. The assumption that $\left|E\right|<\delta\len\left(\omega\right)$ implies that:
\[
\left|\bigcup_{i=1}^{k}\badb_{i,0}\left(\omega\right) \cup E \right| < \left(\delta+\frac{1}{1000A^{4}}\right) \len\left(\omega\right).
\]
By this assumption we use Proposition \ref{prop:existence_of_satisfied_test} with the parameters $m,E,E^{+},E^{-}$, 
and $\alpha_{\mathrm{e}} = \delta+10^{-3} A^{-4}$ in the role of $\alpha$, 
to get an element $\lambda\in\Lambda$ such that the implications from that proposition hold, and set $\lambda_{\mathrm{passed}} = \lambda$.
\subsubsection{Sufficient conditions for Theorem \ref{thm:main_theorem} (1).}
Proposition \ref{prop:existence_of_satisfied_test} says that for any $\left(t,\Delta\right)\in\lambda_{\mathrm{passed}}$:
\begin{enumerate}
\item 
$
\Delta = 
2\left| E^{+} \cap \left\{0,\dots,t\right\} \right| - 
2\left| E^{-} \cap \left\{0,\dots,t\right\} \right| =
X^{\omega}_{t} - X^{\omega'}_{t}
$, 
using \eqref{eq:parameters_for_prop_existence_of_satisfied_test} for the second equality.
\item 
$\left|I'_{t}\cap E\right| < 
	\left(\prod_{r=2}^{m}\left(1+\frac{1}{r^2}\right)\right)\alpha_{\mathrm{e}} \left|I'_{t}\right|$
, where $I'_{t} = \left\{t,\dots,\next\left(\omega,L_{1},t\right)\right\}$
.
\item 
$M^{\mathrm{lower}}_{1} \le \left|I'_{t}\right| \le M^{\mathrm{upper}}_{1}$.
\end{enumerate}

Combining (2) and (3), we get that if:
\begin{equation}
\label{eq:parameter_delta_requirement}
\delta<10^{-3}A^{-4}<\frac{1}{1000M^{\mathrm{upper}}_{1}}
\end{equation}
then:
\[
\alpha_{\mathrm{e}} < 2\cdot10^{-3}A^{-4} < \frac{2}{1000 M^{\mathrm{upper}}_{1}} < \frac{1}{M^{\mathrm{upper}}_{1} \prod_{r=2}^{\infty}\left(1+\frac{1}{r^2}\right)}
\]
and thus $\left|I'_{t}\cap E\right|<1$, and necessarily $I'_{t}\cap E = \emptyset$. Therefore, for each $t\in I'_{t}$, assertion (2) of this theorem holds:
\[
\sigma\left(X^{\omega'}_{t} + \Delta_{k}\right)
=\sigma\left(X^{\omega}_{t}\right)
=x\left(t\right)_{2}
=x'\left(t\right)_{2}.
\]
In the terms of Definition \ref{def:passes_test}, this means that $x'$ passes the test $\lambda$ relative to $\sigma$, and Theorem \ref{thm:main_theorem} (1) holds.
\subsubsection{Sufficient conditions for Theorem \ref{thm:main_theorem} (2) and (3).}
Proposition \ref{prop:existence_of_satisfied_test} with the above parameters implies that:
\[\left|\left\{
X^{\omega}_{t} : t\in I'_{t_{1}} , \left(t_{1},\Delta_{1}\right)\in\lambda
\right\}\right| \ge L_{1} \cdot \left|\lambda\right|\]
which is, by Theorem \ref{thm:main_theorem} (1), equivalent to:
\begin{equation}
\label{eq:reconstructed_scenery}
\left|\lambda\left(\omega'\right)\right|
= \left|\bigcup_{k=1}^{\left|\lambda\right|}\left\{X^{\omega'}_{t}+\Delta_{k}:t\in I'_{k}\right\}\right| 
\ge L_{1} \cdot \left|\lambda\right|
= A \cdot \left|\lambda\right|.
\end{equation}
By Lemma \ref{lem:how_much_psila}, assuming that $B\ge13$, $A\ge20000B$ and $k\ge20A$, and by \eqref{eq:flexible_length_rescaling}:
\begin{align*}
\eqref{eq:reconstructed_scenery} & \ge L_{1}\left|\lambda\right|> \left|\lambda\right|
\ge \left(M^{\mathrm{upper}}_{k}\right)^{\frac{1}{2} - \frac{200}{B}}
\ge \left(\len\left(\omega\right)\right)^{\frac{1}{2} - \frac{200}{B}}
\\&> N^{\left(\frac{1}{2}+\theta\right)\left(\frac{1}{2} - \frac{200}{B}\right)}
> N^{\frac{1}{4} + \frac{\theta}{2} - \frac{200}{B}}
\end{align*}
thus, for Theorem \ref{thm:main_theorem} (2) to hold, we need to require 
\begin{equation}
\label{eq:parameter_theta_requirement}
\frac{1}{4} + \frac{\theta}{2} - \frac{200}{B}>\theta.
\end{equation}

Now, by Proposition \ref{prop:Bound_on_size_of_lambda}, Lemma \ref{lem:how_much_psila} and \eqref{eq:reconstructed_scenery}, 
by assuming that $B\ge16$, $A\ge2^{200}$ and $B\le32A^{\frac{1}{32}}$, and also that $A\ge20000B$ (also $B\ge13$ by these assumptions) we will get:
\begin{align*}
\ln\left|\Lambda_{k}\left(A,B,k\right)\right|
<10^{6}A^{\frac{1}{32}} B_{2}\cdots B_{m}
=10^{6}A^{\frac{1}{32}} \left|\lambda\right|
\le \frac{10^{6}A^{\frac{1}{32}}}{A} \left|\lambda\left(\omega'\right)\right|.
\end{align*}
Thus, for Theorem \ref{thm:main_theorem} (3) to hold, we need to require:
\begin{equation}
\label{eq:parameter_epsilon_requirement}
10^{6} A^{-\frac{31}{32}}<\epsilon
\end{equation}

\subsubsection{Choice of $A,B$}
Let us list the assumptions that we collected in the course of the proof to ensure that Theorem \ref{thm:main_theorem} holds:
\begin{enumerate}
\item 
	$B\ge250$ 	(for Proposition \ref{prop:prob_of_sinuosity}, and also for Proposition \ref{prop:Bound_on_size_of_lambda} and Lemma \ref{lem:how_much_psila})
\item 
	$B>\frac{400}{\frac{1}{2}-\theta}$ 
	(cf. \eqref{eq:parameter_theta_requirement})
\item 
	$A \ge \left(\frac{B}{32}\right)^{32}$ 
		(for Proposition \ref{prop:Bound_on_size_of_lambda})
\item 
	$A > 10^{-\frac{3}{4}} p^{-\frac{1}{4}}$
		(cf. \eqref{eq:parameter_p_requirement})\item 
	$A\ge20000B$
	(for Lemma \ref{lem:how_much_psila})
\item 
    $A > \left(\frac{10^{6}}{\epsilon}\right)^{\frac{32}{31}}$ 
			(cf. \eqref{eq:parameter_epsilon_requirement})
\item 
	$A\ge2^{200}$
	(for Proposition \ref{prop:Bound_on_size_of_lambda}, and also for Proposition \ref{prop:prob_of_sinuosity})
\item 
	$k\ge20A$
	(for Proposition \ref{prop:prob_of_sinuosity} and Lemma \ref{lem:how_much_psila})
\item 
	$k\ge k_{0}$ (condition for \eqref{eq:flexible_length_rescaling})
\item 
	$\delta<10^{-3}A^{-4}$ 
	(cf. \eqref{eq:parameter_delta_requirement}).\end{enumerate}
We may chose values for $A,B$ depending on $p,\theta,\epsilon$ alone, first $B$ such that (1) and (2) hold, and then $A$ such that (3)-(7) hold. 
Now, there exists $N_{0}$ such that if $N>N_{0}$ then (8) and (9) hold (this is the $N_{0}$ of Theorem \ref{thm:main_theorem}, not of Definition \ref{defn:r_all}).
Finally, (10) holds for small enough $\delta$. Thus the proof of Theorem \ref{thm:main_theorem} is complete.

\end{document}